\documentclass[a4paper,11pt]{article}

\pdfoutput=1

\usepackage[a4paper]{geometry}
\usepackage{fullpage,graphicx,mathtools,amssymb,amsthm}
\usepackage[style=alphabetic,backref]{biblatex}
\usepackage[colorlinks=true]{hyperref}
\usepackage[capitalize,nameinlink,noabbrev]{cleveref}

\crefname{subsection}{Subsection}{subsections}
\crefname{subsubsection}{Section}{subsubsections}

\theoremstyle{plain}
\newtheorem{Th}{Theorem}[section]
\newtheorem{lemma}[Th]{Lemma}
\newtheorem{Cor}[Th]{Corollary}
\newtheorem{Prop}[Th]{Proposition}

\theoremstyle{definition}
\newtheorem{Def}[Th]{Definition}
\newtheorem{Conj}[Th]{Conjecture}
\newtheorem{Rem}[Th]{Remark}
\newtheorem{Pro}[Th]{Problem}

\newcommand{\R}{\mathbb{R}}
\newcommand{\C}{\mathbb{C}}
\newcommand{\N}{\mathbb{N}}
\newcommand{\Z}{\mathbb{Z}}

\newcommand{\zv}{\mathbf{z}}
\newcommand{\yv}{\mathbf{y}}

\newcommand{\1}{\boldsymbol{1}}
\newcommand{\0}{\boldsymbol{0}}

\newcommand{\Scal}{\mathcal{S}}
\newcommand{\Pcal}{\mathcal{P}}
\newcommand{\Hcal}{\mathcal{H}}
\newcommand{\Acal}{\mathcal{A}}
\newcommand{\Mcal}{\mathcal{M}}
\newcommand{\Rcal}{\mathcal{R}}

\newcommand{\prb}{\mathbb{P}}
\newcommand{\E}{\mathbb{E}}
\newcommand{\X}{\mathfrak{X}}
\newcommand{\Y}{\mathcal{Y}}

\newcommand{\im}{\operatorname{Im}}
\newcommand{\M}{\mathrm{M}}
\newcommand{\m}{\mathrm{maxroot}}
\newcommand{\ab}{\mathrm{Ab}}
\newcommand{\supp}{\mathrm{supp}}

\newcommand{\ovl}{\overline}
\newcommand{\bld}{\boldsymbol}

\title{\textbf{Paving Property for Real Stable Polynomials and Strongly Rayleigh Processes}}
\author{Kasra Alishahi
\thanks{Department of Mathematical Sciences, Sharif University of Technology. Email: \protect\url{alishahi@sharif.edu}.}
\and
Milad Barzegar
\thanks{Department of Mathematical Sciences, Sharif University of Technology. Email: \protect\url{milad.barzegar@sharif.edu}.}
}
\date{}

\begin{document}
\maketitle

\begin{abstract}

One of the equivalent formulations of the Kadison-Singer problem which was resolved in 2013 by Marcus, Spielman and Srivastava, is the ``paving conjecture". Roughly speaking, the paving conjecture states that every positive semi-definite contraction with small diagonal entries can be ``paved" by a small number of principal submatrices with small operator norms. We extend this result to real stable polynomials. We will prove that assuming mild conditions on the leading coefficients of a multi-affine real stable polynomial, it is possible to partition the set of variables to a small number of subsets such that the roots of the ``restrictions" of the polynomial to each set of variables are small.

We will use this generalized paving theorem to show that for every strongly Rayleigh point process, it is possible to partition the underlying space into a small number of subsets such that the points of the restrictions of the point process to each subset are ``weakly correlated". This result is intuitively appealing since it implies that the repulsive force among the points of a negatively dependent point process cannot be strong everywhere. To prove this result, we will introduce the notion of the kernel polynomial for strongly Rayleigh processes. This notion is a generalization of the notion of the kernel of determinantal processes and provides a unified framework for studying these two families of point processes. We will also prove an entropy lower bound for strongly Rayleigh processes in terms of the roots of the kernel polynomial.

\end{abstract}

\section{Introduction}

In 1959, Richard V. Kadison and Igor M. Singer \cite{Kadison1959Extensions} raised the question whether every pure state on the algebra of bounded diagonal operators on $\ell^2(\N)$ has a unique extension to a state on the algebra of all bounded operators on $\ell^2(\N)$. This problem has come to be known as the \textit{Kadison-Singer problem}. Over the next 54 years, this problem attracted a significant amount of research until it was resolved in the affirmative in 2013 by Adam Marcus, Daniel Spielman and Nikhil Srivastava \cite{Marcus2015InterlacingII}.

One important aspect of the Kadison-Singer problem is that it has been shown to be equivalent to a large number of problems in various fields. One of these equivalent formulations which will be our main focus, is as follows.

\begin{Pro}\label{Pro:PavingConjecture}
Let $\varepsilon \in (0,1)$. Does there exist $r \in \N$ such that every Hermitian matrix $A$ whose diagonal entries are zero can be $(r,\varepsilon)$-paved, i.e., there are diagonal projections $P_1,\dots,P_r$ such that $\sum_{i=1}^r P_i = I$ and
\begin{align*}
\forall i \in [r] \ : \ \| P_i A P_i \|_{op} \leq \varepsilon \, \| A \|_{op},
\end{align*}
where $[r] = \{1,\dots,r\}$ and $\| \cdot \|_{op}$ denotes the operator norm.
\end{Pro}

This formulation of the Kadison-Singer problem, which is known as the \textit{paving problem} (or the \textit{paving conjecture} for the assertion that the answer to the above question is ``yes"), was discovered by Joel Anderson \cite{Anderson1979Extensions}. Anderson showed that the answer to the paving problem is positive if and only if the answer to the Kadison-Singer problem is positive. For a background on the Kadison-Singer problem and its equivalent formulations see \cite{Bownik2018Kadison} and the references therein.

Marcus et al. \cite{Marcus2015InterlacingII} proved a stronger version of ``Weaver's vector balancing formulation" of the Kadison-Singer problem (see \cite{Bownik2018Kadison}) using the ``method of interlacing families". This method was first introduced in \cite{Marcus2015InterlacingI} and provides a technique for proving the existence of certain combinatorial objects. We will review interlacing families in \cref{Subsec:InterlacingFamilies}. The application of this method results in an analysis of the locations of the roots of a real stable polynomial. The ``multivariate barrier method", introduced in \cite{Marcus2015InterlacingII}, provides a framework for such an analysis.

Marcus et al. \cite{Marcus2015InterlacingII} obtained the following paving bound for positive semi-definite contractions with bounded diagonal entries.

\begin{Th}\label{Th:MSSBoundedDiagonalPaving}
Let $\alpha$ be a positive number and $r$ be an integer such that $r \geq 2$. For every positive semi-definite contraction $A \in \Mcal_n(\C)$ with diagonal entries at most $\alpha$, there are diagonal projections $P_1,\dots,P_r \in \Mcal_n(\C)$ such that $\sum_{i=1}^r P_i = I_n$ and
\begin{align*}
\forall i \in [r] \ : \ \| P_i A P_i \|_{op} \leq \bigg( \sqrt{\dfrac{1}{r}} + \sqrt{\alpha} \bigg)^2.
\end{align*}
\end{Th}

Leake and Ravichandran \cite{Leake2020Mixed} adapted the methods of \cite{Marcus2015InterlacingII} to directly prove the paving conjecture and as a result, they got sharper paving bounds. They used the method of interlacing families in conjunction with a modified version of the multivariate barrier method. Their result is as follows.

\begin{Th}\label{Th:LRMatrixPaving}
Let $r \in \Z$ and $\alpha \in \R$ such that $r \geq 2$ and $0 < \alpha \leq (r-1)^2/r^2$. For every positive semi-definite contraction $A \in \Mcal_n(\C)$ with diagonal entries at most $\alpha$, there are diagonal projections $P_1,\dots,P_r \in \Mcal_n(\C)$ such that $\sum_{i=1}^r P_i = I$ and
\begin{align*}
\forall i \in [r] \ : \ \| P_i A P_i \|_{op} \leq \bigg( \sqrt{\dfrac{1}{r} - \dfrac{\alpha}{r-1}} + \sqrt{\alpha} \bigg)^2.
\end{align*}
\end{Th}

In \cref{Sec:PolynomialPaving}, we will show that the arguments of \cite{Leake2020Mixed} extend to \textit{real stable polynomials} and we will obtain a generalization of the above theorem. It is worthwhile to mention that two other generalizations of the Kadison-Singer problem appear in \cite{Anari2014Kadison} and \cite{Branden2018Hyperbolic}, both of which are through the main result of \cite{Marcus2015InterlacingII}.

A polynomial $p \in \C[z_1,\dots,z_n]$ is \textit{stable} if it has no roots in $\mathbb{H}^n$, where $\mathbb{H}$ is the open upper half-plane, and it is \textit{real stable} if, in addition, its coefficients are real. We will review stable polynomials in \cref{Subsec:StablePolynomials}. 

Before stating our result, let us fix some notations. When $n$ is specified in the context, we will use $\zv = (z_1,\dots,z_n)$. Let $\partial_i := \partial / \partial z_i$ and for every $I \subset [n]$, define $\zv^I = \prod_{i \in I} z_i$ and $\partial^I = \prod_{i \in I} \partial_i$. For $p \in \C[z_1,\dots,z_n]$, we will use $\ovl{p}$ to denote its \textit{diagonalization}  defined by $\ovl{p}(x) = p(x,\dots,x)$. For a real rooted polynomial $p$, we will denote its maximum root by $\m(p)$. Our generalization of \cref{Th:LRMatrixPaving} is as follows.

\begin{Th}\label{Th:PolynomialPaving}
Let $r \in \Z$ and $\alpha \in \R$ such that $r \geq 2$ and $0 < \alpha \leq (r-1)^2/r^2$. Assume that $g \in \R[z_1,\dots,z_n]$ is a multi-affine real stable polynomial and $g(\zv) = \sum_{A \subseteq [n]} a_A \, \zv^{A^c}$. If all the roots of $\ovl{g}$ are in the interval $[0,1]$, $a_\emptyset=1$ and $|a_{\{i\}}| \leq \alpha$ for $i = 1,\dots,n$, then there exists a partition $\{S_1,\dots,S_r\}$ of $[n]$ such that
\begin{align*}
\forall i \in [r] \ : \ \m \big( \, \ovl{\partial^{S_i^c}g} \, \big) \leq \bigg( \sqrt{\dfrac{1}{r} - \dfrac{\alpha}{r-1}} + \sqrt{\alpha} \bigg)^2.
\end{align*}
\end{Th}

We will prove the above theorem in \cref{Subsec:ProofOfPolynomialPaving}. To deduce \cref{Th:LRMatrixPaving} from \cref{Th:PolynomialPaving}, we need the notion of \textit{multivariate characteristic polynomial} of a matrix. The multivariate characteristic polynomials of a matrix $A \in \Mcal_n(\C)$, denoted $\chi[A]$, is defined by $\chi[A](\zv) = \det[Z-A]$, where $Z = \mathrm{Diag}(z_1,\dots,z_n)$. The multivariate characteristic polynomial of Hermitian matrices are real stable (see the remarks following \cref{Prop:R.S.P.Example}).

Let $A$ be as in \cref{Th:LRMatrixPaving}. Note that the coefficient of the monomial $z_1 \dots z_n$ in $\chi[A]$ is equal to 1. Also, for each $i \in [n]$, the coefficient of the monomial $z_1 \dots z_{i-1} z_{i+1} \dots z_n$ is equal to $A_{i,i}$ and so its absolute value is less than $\alpha$. Since $\ovl{\chi[A]}$ is the characteristic polynomial of $A$ and $A$ is a positive semi-definite contraction, all its roots are in the interval $[0,1]$. Therefore, by \cref{Th:PolynomialPaving}, there exists a partition $\{S_1,\dots,S_r\}$ of $[n]$ such that
\begin{align*}
\forall i \in [r] \ : \ \m \big( \, \ovl{\partial^{S_i^c} \chi[A]} \, \big) \leq \bigg( \sqrt{\dfrac{1}{r} - \dfrac{\alpha}{r-1}} + \sqrt{\alpha} \bigg)^2.
\end{align*}
Let $P_i \in \Mcal_n(\C)$ be the diagonal matrix whose $k$-th diagonal entry is equal to 1 if $k \in S_i$ and is equal to 0 if $k \not\in S_i$. Since $\{S_1,\dots,S_r\}$ is a partition of $[n]$, we have $\sum_{i=1}^r P_i = I$. Now, \cref{Th:LRMatrixPaving} follows since
\begin{align*}
\|P_iAP_i\|_{op} = \m\big( \, \ovl{\chi[P_iAP_i]} \, \big) \quad \text{and} \quad \chi[P_iAP_i] = \partial^{S_i^c} \chi[A].
\end{align*}

We will use \cref{Th:PolynomialPaving} to prove a ``paving property" for strongly Rayleigh point processes. A point process $\X$ on $[n]$, i.e. a random subset of $[n]$, is \textit{strongly Rayleigh} if its probability generating polynomial, defined as
\begin{align*}
f_\X(\zv) = \sum_{A \subseteq [n]} \prb (\X = A) \, \zv^A,
\end{align*}
is real stable.

Robin Pemantle in \cite{Pemantle2000Towards} emphasized the need for a theory of negative dependence which would take shape around an appropriate notion of negative dependence. The strong Rayleigh property was introduced by Borcea, Br\"{a}nd\'{e}n and Liggett \cite{Borcea2009Negative} as this appropriate notion. Strongly Rayleigh point processes have many useful properties including negative association which is the strongest form of negative dependence, and they cover several well-known examples of negatively dependent processes, most notably discrete ``determinantal processes". These processes have also found numerous application; see, for example, \cite{Gharan2011Randomized, Branden2012Negative, Pemantle2014Concentration, Anari2014Kadison, Anari2016Monte, Ghosh2017Multivariate}.

We will review strongly Rayleigh processes in \cref{Subsec:S.R.Processes}. We will also introduce the notion of \textit{kernel polynomial} for strongly Rayleigh processes which plays a role similar to the kernel of determinantal processes and provides a unified framework for studying strongly Rayleigh and determinantal processes.

We need the notion of entropy in order to state the paving property of strongly Rayleigh processes. Recall that the entropy of a random element $X$ from a finite set $S$, denoted $H(X)$, is defined by
\begin{align*}
H(X) = - \sum_{x \in S} \, \prb(X=x) \, \log\!\big(\prb(X=x)\big),
\end{align*}
where the logarithms are taken in base 2. We use $h(p)$ to the denote the entropy of a Bernoulli random variable $X$ with $\prb(X=1) = p$. The paving property for strongly Rayleigh processes is as follows.

\begin{Th}\label{Th:S.R.P.PavingEntropyVersion}
For each positive number $\delta$, there exists an integer $r$ such that for every strongly Rayleigh process $\X$ on any space $S$, it is possible to partition $S$ into $r$ subsets $S_1,\dots,S_r$ such that
\begin{align*}
\forall i \in [r] \ : \ \Bigg| \dfrac{1}{|S_i|} H(\X \cap S_i) - \dfrac{1}{|S_i|} \sum_{j \in S_i} h(p_j) \Bigg| < \delta,
\end{align*}
where $|S_i|$ denotes the size of $S_i$ and $p_j = \prb(j \in \X)$.
\end{Th}

Note that $r$ does not depend on the size of $S$. This implies that for strongly Rayleigh processes on large enough spaces, the underlying space can be partitioned into a small number of sets such that the entropy per particle of the restrictions of the process to each part is close to that of its independent version. We interpret this as the points of each restriction being ``almost independent". This is in line with the behavior that we expect from a point process with repulsion; that is, we expect that the correlation structure of the points of such a process is constrained in the sense that all its points cannot simultaneously be strongly correlated. We will discuss this phenomenon in more detail in \cref{Subsec:RepellingProperty}.

We will prove \cref{Th:S.R.P.PavingEntropyVersion} in \cref{Sec:ProofOfS.R.P.PavingEntropyVersion}. To this end, we will apply a slightly modified version of \cref{Th:PolynomialPaving}, presented in \cref{Subsec:ProofOfZeroDiagonalPolynomialPaving}, to the kernel polynomial of $\X$. This will give us a partition of the underlying space with the property that the roots of ``centered versions" of the kernels of the restricted processes are simultaneously small. We will translate this algebraic condition to an entropy inequality via the connection between stability and ``hyperbolicity" and exploiting the majorization properties of hyperbolic polynomials. We will also need an entropy estimation in terms of the roots of its kernel polynomials which we will present in \cref{Subsec:EntropyLowerBound}.

This entropy bound is interesting on its own. Since the correlation structure of a strongly Rayleigh process is constrained, we expect that its entropy cannot be too small. \cite[Corollary 5.6]{Anari2018LogConcave} provides a lower bound for the entropy of strongly Rayleigh processes in terms of the entropy of its marginals. We will prove a lower bound for the entropy in terms of the roots of the kernel polynomial (see \cref{Th:EntropyLowerBound}). We will demonstrate in \cref{Subsec:S.R.Processes} that the roots of of the kernel polynomial play a similar role to the eigenvalues of the kernel of determinantal processes. Motivated by this comparison, one can ask whether there is a ``probabilistic interpretation" of the roots of the kernel polynomial. We will propose a conjecture which can be regarded as a first step in formalizing this question.

\section{Preliminaries}

We will use $\1$ to denote the vector of all $1$'s, i.e. $\1 = (1,\dots,1)$. Similarly, $\0 := (0,\dots,0)$. We will denote the $i$-th entry of a vector $v$ by $v_i$. For $v,w \in \R^n$, we will use $v \geq w$ when $v_i \geq w_i$, for all $i \in [n]$. For $p \in \C[z_1,\dots,z_n]$ let $\deg(p)$ denote the degree of $p$ and $\deg_j(p)$ denote the degree of $p$ in $z_j$. Also, for $v \in \N^n$, we will use $[\zv^v]_p$ to denote the coefficient of the monomial $\zv^v$ in $p$, where $\zv^v := \prod_{i} z_i^{v_i}$. For a real rooted polynomial $p \in \R[x]$, we will use $\lambda(p)$ to denote the non-increasing vector of the its roots and $\lambda_i(p)$ to denote its $i$-th largest root.

\subsection{Interlacing Families}\label{Subsec:InterlacingFamilies}

\begin{Def}
Two non-increasing sequences $(\alpha_1,\dots,\alpha_m)$ and $(\beta_1,\dots,\beta_n)$ are \textit{interlacing} if they alternate, namely
\begin{align*}
\alpha_1 \geq \beta_1 \geq \alpha_2 \geq \beta_2 \geq \dots \qquad \text{or} \qquad \beta_1 \geq \alpha_1 \geq \beta_2 \geq \alpha_2 \geq \dots,
\end{align*}
in which case we clearly must have $|m - n| \leq 1$. We say that $(\alpha_1,\dots,\alpha_{n-1})$ \textit{interlaces} $(\beta_1,\dots,\beta_n)$ if
\begin{align*}
\beta_1 \geq \alpha_1 \geq \beta_2 \geq \alpha_2 \geq \dots \geq \alpha_{n-1} \geq \beta_n.
\end{align*}

Two real rooted polynomials $p$ and $q$ are \textit{interlacing} if their roots are interlacing and $p$ \textit{interlaces} $q$ if $\deg(p) = \deg(q)-1$ and $\lambda(p)$ interlaces $\lambda(q)$. We also assume that zero interlaces and is interlaced by every real rooted polynomial.
\end{Def}

Polynomials $p_1,\dots,p_k$ of the same degree have a \textit{common interlacer} if there is a polynomial $q$ that interlaces all of them. A fundamental property of polynomials with a common interlacer is as follows.

\begin{Prop}[Lemma 4.2 of \cite{Marcus2015InterlacingI}]\label{Prop:InterlacingSequence}
Let $p_1,\dots,p_k$ be real rooted polynomials of the same degree with positive leading coefficients. If $p_1,\dots,p_k$ have a common interlacer, then their summation, denoted by $p_\emptyset$, is real rooted and there exists $i \in [k]$ such that the largest root of $p_i$ is less than or equal to the largest root of $p_\emptyset$.
\end{Prop}

The following theorem is often used to prove the existence of a common interlacer.

\begin{Th}[Theorem 2.1 of \cite{Dedieu1992Obreschkoff}]\label{Th:GeneralizedObreshkoff'sTheorem}
Let $p_1,\dots,p_k$ be univariate polynomials of the same degree with positive leading coefficients. Then $p_1,\dots,p_k$ have a common interlacer if and only if all convex combinations of $p_1,\dots,p_k$, namely all $\sum_{i=1}^k \alpha_i p_i$ with $\alpha_i \geq 0$ and $\sum_{i=1}^k \alpha_i = 1$, are real rooted.
\end{Th}

Marcus et al. \cite{Marcus2015InterlacingI} generalized \cref{Prop:InterlacingSequence} to ``interlacing families".

\begin{Def}\label{Def:InterlacingFamily}
A family of polynomials with positive leading coefficients is an \textit{interlacing family} if it is possible to attach them to the nodes of a rooted tree in a way that the following conditions hold:
\begin{enumerate}
\item
Each polynomial at a (non-leaf) node is equal to the sum of the polynomials attached to its children.
\item
The polynomials at sibling nodes (nodes with the same parent) have a common interlacer.
\end{enumerate}
\end{Def}

Note that the polynomial attached to the root is automatically equal to the sum of the polynomials attached to the leaves. The following theorem is a generalization of \cref{Prop:InterlacingSequence}.

\begin{Th}[Theorem 4.4 of \cite{Marcus2015InterlacingI}]\label{Th:InterlacingFamilies}
Let $T$ be a rooted tree with root $r$. If univariate polynomials $(p_n)_{n \in T}$ form an interlacing family, then the polynomial attached to the root, denoted $p_r$, is real rooted and there exists a leaf $n \in T$ such that
\begin{align*}
\m ( p_n ) \leq \m ( p_r ).
\end{align*}
\end{Th}

\subsection{Stable Polynomials}\label{Subsec:StablePolynomials}

Stable polynomials are a natural multivariate generalization of real rooted polynomials. These polynomials have many nice algebraic and geometric properties. In this subsection, we summarize some of these properties that we need for later use. See the surveys \cite{Pemantle2011Hyperbolicity} and \cite{Wagner2011Multivariate} for a thorough overview of this subject.

\begin{Def}
A polynomial $p \in \C[z_1,\dots,z_n]$ is \textit{stable} if
\begin{align*}
\im(z_1)>0 , \dots , \im(z_n)>0 \, \Longrightarrow \, p(z_1,\dots,z_n) \neq 0.
\end{align*}
$p$ is \textit{real stable} if, in addition, its coefficients are real. We use $\Hcal_n(\C)$ and $\Hcal_n(\R)$ to denote the set of $n$-variate stable and real stable polynomials, respectively.
\end{Def}

Note that a univariate real polynomial is (real) stable if and only if it is real rooted. The following proposition is an immediate consequence of the definition.

\begin{Prop}\label{Prop:EquivalentOfStablitiy}
A polynomial $p \in \C[z_1,\dots,z_n]$ is stable (real stable, respectively) if and only if for every $\alpha \in \R^n$ and $v \in \R_+^n$, the univariate polynomial $t \mapsto p(tv + \alpha)$ is stable (real stable, respectively).
\end{Prop}

Determinantal polynomials are the most important examples of real stable polynomials.

\begin{Prop}[Proposition 1.12 of \cite{Borcea2010Multivariate}]\label{Prop:R.S.P.Example}
If $B \in \Mcal_n(\C)$ is a Hermitian matrix and $A_1,\dots,A_n \in \Mcal_n(\C)$ are positive semi-definite, then the polynomial $\det(B+z_1A_1+\dots+z_mA_m)$ is either identically zero or real stable.
\end{Prop}

It follows from the above proposition that for every Hermitian matrix $K \in \Mcal_n(\C)$, the polynomial $\det(Z-K)$, where $Z:=\mathrm{Diag}(z_1,\dots,z_n)$, is real stable. This polynomial is the \textit{multivariate characteristic polynomial} of $K$ and we denote it by $\chi[K](\zv)$.

The class of (real) stable polynomials is closed under several elementary operations. Some of these closure properties are summarized in the following proposition. See \cite{Borcea2010Multivariate} and \cite{Borcea2008Application} for the proofs.

\begin{Prop}\label{Prop:R.S.P.Properties}
If $p$ is a (real) stable polynomial in $n$ variables, then
\begin{enumerate}
\item
$\partial_ip$ is identically zero or (real) stable for $i \in [n]$;

\item
$p(z_1,\dots,z_{i-1},\beta,z_{i+1},\dots,z_n)$ is identically zero or (real) stable for $i \in [n]$ and $\beta \in \R$;

\item
$p(z_1,\dots,z_{i-1},z_j,z_{i+1},\dots,z_n)$ is (real) stable for distinct $i,j \in [n]$. In particular $\ovl{p}$ is (real) stable. Recall that $\ovl{p}(x) = p(x,\dots,x)$.

\item
If $p$ is real stable then $z_1^{d_1} \dots z_n^{d_n} \, p(\gamma_1z_1^{-1},\dots,\gamma_nz_n^{-1})$ is real stable for $\pm(\gamma_1,\dots,\gamma_n) \in \R_+^n$.
\end{enumerate}
\end{Prop}

In the previous subsection we defined the notion of interlacing for real rooted polynomials. This notion has been generalized by Borcea and Br\"{a}nd\'{e}n to the multivariate case. Let $p,q \in \Hcal_1(\R)$. It is known that if $p$ and $q$ are interlacing, then the Wronskian, defined by $W[p,q] = pq'-p'q$, is either non-negative or non-positive on the real line. We say that $q$ is in \textit{proper position} with respect to $p$, denoted $q \ll p$, if $p$ and $q$ are interlacing and $W[p,q] \leq 0$. Note that if $\deg(q) < \deg(p)$ and they are in proper position, then $\deg(q) = \deg(p)-1$ and $q$ interlaces $p$. Also, if the leading coefficients of $p$ and $q$ have the same sign, then $q \ll p$ if and only if
\begin{align*}
\lambda_1 \geq \gamma_1 \geq \lambda_2 \geq \gamma_2 \geq \dots,
\end{align*}
where $\lambda_1 \geq \lambda_2 \geq \dots$ are the roots of $p$ and $\gamma_1 \geq \gamma_2 \geq \dots$ are the roots of $q$. Also note that $q \ll p$ if and only if $-p \ll q$.

The notion of proper position is generalized as follows.

\begin{Def}\label{Def:ProperPosition}
Let $p,q \in \R[z_1,\dots,z_n]$. We say that $q$ is in \textit{proper position} with respect to $p$, denoted $q \ll p$, if for all $\alpha \in \R^n$ and $v \in \R_+^n$ the univariate polynomial $q(tv+\alpha)$ is in proper position with respect to $p(tv+\alpha)$.
\end{Def}

It follows from \cref{Prop:EquivalentOfStablitiy} and the Hermite-Biehler theorem (see \cite{Rahman2002Analytic}) that $q \ll p$ if and only if $p +i q \in \Hcal_n(\C)$. A well known example of polynomials in proper position is $p$ and $\partial_i p$ for every $p \in \Hcal_n(\R)$ and $i \in [n]$, where we have $\partial_i p \ll p$ (see, e.g., \cite{Borcea2010Multivariate}). An important consequence of the definition is that $q \ll p$ implies $\ovl{q} \ll \ovl{p}$. We will use this fact several times.

A polynomial $p \in \C[z_1,\dots,z_n]$ is called \textit{multi-affine} if $\deg_i(p) \leq 1$ for all $i \in [n]$. The class of multi-affine real stable polynomials are of special importance. In this paper we are mainly interested in such polynomials. In the rest of this subsection, we will prove some results concerning multi-affine real stable polynomials that will be useful for us later on.

\begin{Prop}\label{Prop:ExamplesOfInterlacing}
Let $p \in \Hcal_n(\R)$ be multi-affine and $p = r + z_n s$, where $r,s \in \R[z_1,\dots,z_{n-1}]$. We have $s \ll p$ and $s \ll r$.
\end{Prop}

\begin{proof}
Since $s = \partial_n p$ and $\partial_n p \ll p$, we have $s \ll p$. Since $p|_{z_n=i}$ is stable and $p|_{z_n=i} = r + is$, we have $s \ll r$.
\end{proof}

For a multi-affine polynomial $p \in \R[z_1,\dots,z_n]$ with $p (\zv) = \sum_{A \subseteq [n]} a_A \, \zv^A$, define its \textit{support}, denoted by $\supp(p)$, to be the set $\big\{A \subseteq [n] : a_A \neq 0\big\}$.

\begin{Prop}[Corollary 3.7 of \cite{Branden2007Polynomials}]\label{Prop:SuppPositiveCoeffs}
Let $p$ be a multi-affine real stable polynomial with non-negative coefficients. If $A \subseteq C \subseteq B$ and $A,B \in \supp(p)$, then $C \in \supp(p)$.
\end{Prop}

\begin{lemma}\label{Lem:PositiveRootedPolynomialCoefficientSigns}
Let $p \in \R[z_1,\dots,z_n]$ with $p(\zv) = \sum_{A \subseteq [n]} a_A \, \zv^A$ be a multi-affine real stable polynomial and $a_{[n]} > 0$. If all the roots of $\ovl{p}$ are non-negative, then $(-1)^{n-|A|} \, a_A \geq 0$ for every $A \subseteq [n]$.
\end{lemma}

\begin{proof}
By \cref{Prop:ExamplesOfInterlacing}, for every $A \subseteq [n]$ with $A = \{i_1,\dots,i_k\}$ we have
\begin{align*}
\ovl{p} \gg \ovl{\partial^{\{i_1\}}p} \gg \ovl{\partial^{\{i_1,i_2\}}p} \gg \dots \gg \ovl{\partial^Ap}.
\end{align*}
Since $p$ is multi-affine and $a_{[n]}>0$, the degree of each polynomial in the above sequence is one less than the degree of the polynomial on its left. Therefore, each polynomial interlaces the polynomial on its left. This implies that all the roots of $\ovl{\partial^Ap}$ are non-negative. Since $\ovl{\partial^Ap}$ is of degree $n-|A|$ and its leading coefficient is positive, either $\ovl{\partial^Ap}(0) = 0$ or $(-1)^{n-|A|} \, \partial^Ap(0) \geq 0$. The lemma follows since $\partial^Ap(0) = a_A$.
\end{proof}

We will use the following corollary several times.

\begin{Cor}\label{Cor:SuppPositiveNonnegativeRoots}
Let $p$ be a multi-affine real stable polynomial and all the roots of $\ovl{p}$ be non-negative. If $A \subseteq C \subseteq B$ and $A,B \in \supp(p)$, then $C \in \supp(p)$.
\end{Cor}

\begin{proof}
Define $q(\zv) = (-1)^n \, p(-\zv)$. Note that $\supp(p) = \supp(q)$. By \cref{Lem:PositiveRootedPolynomialCoefficientSigns}, $q$ has non-negative coefficients. The result follows from \cref{Prop:SuppPositiveCoeffs}.
\end{proof}

\section{Paving Property for Real Stable Polynomials}\label{Sec:PolynomialPaving}

In \cref{Subsec:ProofOfPolynomialPaving}, we will prove \cref{Th:PolynomialPaving}. We will adapt the arguments that Leake and Ravichandran \cite{Leake2020Mixed} use to prove \cref{Th:LRMatrixPaving}. In \cref{Subsec:ProofOfZeroDiagonalPolynomialPaving}, we will present a slightly modified version of \cref{Th:PolynomialPaving}. We need the second version in the proof of our probabilistic paving property, \cref{Th:S.R.P.PavingEntropyVersion}.

\subsection{Paving Property for Polynomials, First Version}\label{Subsec:ProofOfPolynomialPaving}

Let $\Pcal_r(n)$ denote the set of all partitions of $[n]$ into $r$, possibly empty, subsets. For a polynomial $g \in \R[z_1,\dots,z_n]$ and $\Scal \in \Pcal_r(n)$ with $\Scal = \{S_1,\dots,S_r\}$, define $g_\Scal \in \R[x]$ as
\begin{align*}
g_\Scal = \prod_{i=1}^r \, \ovl{\partial^{S_i^c}g}
\end{align*}
and $g_r \in \R[x]$ as
\begin{align*}
g_r = \sum_{\Scal \, \in \, \Pcal_r(n)} g_\Scal.
\end{align*}

We will prove the following theorem in \cref{Subsec:ProofOfExistenceOfGoodPartition}.

\begin{Th}\label{Th:ExistenceOfGoodPartition}
Let $g$ be as in \cref{Th:PolynomialPaving}. The polynomial $g_r$ is real rooted and there exists a partition $\Scal \in \Pcal_r(n)$ such that
\begin{align*}
\m ( g_\Scal ) \leq \m ( g_r ).
\end{align*}
\end{Th}

To prove this theorem, we will show that there exists an interlacing family such that the polynomials in $\{g_\Scal : \Scal \in \Pcal_r(n)\}$ are attached to the leaves. See \cref{Subsec:InterlacingFamilies} for an overview of interlacing families.

Note that for $\Scal \in \Pcal_r(n)$ with $\Scal = \{S_1,\dots,S_r\}$, we have
\begin{align*}
\m ( g_\Scal ) = \max_{i \in [r]} \Big( \m \big( \, \ovl{\partial^{S_i^c}g} \, \big) \Big).
\end{align*}
Therefore, for the partition $\{S_1,\dots,S_r\} \in \Pcal_r(n)$ given by \cref{Th:ExistenceOfGoodPartition}, we have
\begin{align*}
\forall i \in [r] \ : \ \m \big( \, \ovl{\partial^{S_i^c}g} \, \big) \leq \m ( g_r ).
\end{align*}
Thus to prove \cref{Th:PolynomialPaving}, it is sufficient to show that
\begin{align}\label{Eq:UpperBoundForMaxrootOfg_r}
\m(g_r) \leq \bigg( \sqrt{\dfrac{1}{r} - \dfrac{\alpha}{r-1}} + \sqrt{\alpha} \bigg)^2.
\end{align}
We will prove this inequality in \cref{Subsec:ProofOfUpperBoundForMaxrootOfg_r} using Leake-Ravichandran's version of ``multivariate barrier method" introduced in \cite{Leake2020Mixed}. This version of the barrier method gives (upper) bounds for the largest root of partial derivatives of a stable polynomial. The following proposition provides such an expression for $g_r$.

\begin{Prop}\label{Prop:CombinatorialExpressionOfg_r}
If $g$ is as in \cref{Th:PolynomialPaving}, then
\begin{align*}
g_r(x) = \bigg( \dfrac{1}{(r-1)!} \bigg)^n \, \Bigg[ \bigg( \prod_{i=1}^n \partial_i^{r-1} \bigg) g(\zv)^r \Bigg] \Bigg|_{\zv = x\1}.
\end{align*}
\end{Prop}

\begin{proof}
We proceed as in the proof of \cite[Lemma 3.1]{Leake2020Mixed}. Using the product rule and the assumption that $g$ is multi-affine, we get
\begin{align*}
\bigg( \prod_{i=1}^n \partial_i^{r-1} \bigg) g(\zv)^r = \big( (r-1)! \big)^n \!\!\! \sum_{(A_1,\dots,A_r) \in \Acal} \; \prod_{i=1}^r \big( \partial^{A_i} g(\zv) \big),
\end{align*}
where $\Acal$ is the collection of all $r$-tuples $(A_1,\dots,A_r)$ of subsets of $[n]$ such that each element of $[n]$ occurs exactly in $r-1$ of $A_i$'s. This is equivalent to $A_1^c,\dots,A_r^c$ being a partition of $[n]$, namely $(A_1^c,\dots,A_r^c) \in \Pcal_r(n)$. Therefore,
\begin{align*}
\bigg( \dfrac{1}{(r-1)!} \bigg)^n \, \Bigg[ \bigg( \prod_{i=1}^n \partial_i^{r-1} \bigg) g(\zv)^r \Bigg]
&= \sum_{(S_1,\dots,S_r) \, \in \, \Pcal_r(n)} \; \prod_{i=1}^r \big( \partial^{S_i^c} g(\zv) \big),
\end{align*}
from which the proposition follows by setting $\zv = x \1$.
\end{proof}

\subsubsection{Interlacing Families: Proof of \texorpdfstring{\cref{Th:ExistenceOfGoodPartition}}{Theorem 3.1}}\label{Subsec:ProofOfExistenceOfGoodPartition}

We proceed as in Section 2 of \cite{Leake2020Mixed}. We will present an interlacing family of polynomials in which the set of leaf-polynomials is $\{g_\Scal : \Scal \in \Pcal_r(n)\}$. Then, \cref{Th:ExistenceOfGoodPartition} follows from \cref{Th:InterlacingFamilies} since, by definition, for such an interlacing family the polynomial attached to the root will automatically be $g_r$. For the sake of simplicity, we will only demonstrate the case $r=2$. The general case is similar and will be briefly discussed in \cref{Rem:Generalr}.

We claim that the following family of polynomials is the appropriate interlacing family.

\begin{Def}\label{Def:OurInterlacingFamily}
Let $T$ be a perfect binary tree with height $n$, namely $T$ is a rooted tree with height $n$ such that each node has exactly two children. Index the nodes of $T$ as follows:
\begin{itemize}
\item
For each $k \in [n]$, the nodes with depth $k$ are indexed by (ordered) partitions of $[k]$ into two subsets, namely ordered pairs $(S,T)$ with $S \sqcup T = [k]$.\footnote{We use the notation $S \sqcup T$ to stress that the sets $S$ and $T$ are disjoint.} The root is indexed by $\emptyset$.
\item
For $k=0,\dots,n-1$, the children of a node $(S,T)$ at depth $k$ are $\big(S\cup\{k+1\} , T\big)$ and $\big(S , T\cup\{k+1\}\big)$.
\end{itemize}
Now, denote the polynomial attached to the node $(S,T)$ at depth $k$ by $q_k(S,T)$. For each leaf $(S,T)$ with $S \sqcup T = [n]$, set
\begin{align*}
q_n(S,T) = g_{\{S,T\}} = \big(\, \ovl{\partial^{S^c}g} \,\big) \big(\, \ovl{\partial^{T^c}g} \,\big).
\end{align*}
The polynomials attached to the other nodes are set to be equal to the sum of the polynomials attached to their children.
\end{Def}

We can compute the attached polynomials. For a node $(S,T)$ at level $k$, we have 
\begin{align}\label{Eq:q_k}
q_k(S,T)(x)
&= \sum_{U \sqcup V = [k+1,n]} q_n(S \sqcup U , T \sqcup V)(x) \nonumber \\
&= \sum_{U \sqcup V = [k+1,n]} \ovl{\partial^{(S \sqcup U)^c} g}(x) \ \ovl{\partial^{(T \sqcup V)^c} g}(x)  \nonumber \\
&= \Bigg[ \sum_{U \sqcup V = [k+1,n]} \partial_\zv^{(T \sqcup V)^c} \partial_\yv^{(S \sqcup U)^c} \, g(\zv) g(\yv)  \Bigg] \Bigg|_{\zv=\yv=x\1} \nonumber \\
&= \Bigg[ \sum_{U \sqcup V = [k+1,n]} \partial_\zv^{S \sqcup U} \partial_\yv^{T \sqcup V} \, g(\zv) g(\yv)  \Bigg] \Bigg|_{\zv=\yv=x\1} \nonumber \\
&= \Bigg[ \partial_\zv^S \partial_\yv^T \sum_{U \sqcup V = [k+1,n]} \partial_\zv^U \partial_\yv^V \, g(\zv) g(\yv)  \Bigg] \Bigg|_{\zv=\yv=x\1} \nonumber \\
&= \Bigg[ \partial_\zv^S \partial_\yv^T \prod_{i=k+1}^n \Big( \partial_{z_i}+\partial_{y_i} \Big) \, g(\zv) g(\yv)  \Bigg] \Bigg|_{\zv=\yv=x\1},
\end{align}
where $\zv = (z_1,\dots,z_n)$ and $\yv = (y_1,\dots,y_n)$.

\begin{lemma}\label{Lem:OurInterlacingFamilyIsInterlacing}
The family of polynomials given by \cref{Def:OurInterlacingFamily} is an interlacing family.
\end{lemma}

\begin{proof}
The first condition in the definition of interlacing families (\cref{Def:InterlacingFamily}) is satisfied by construction. We show that the polynomials attached to the children of each node have a common interlacer. Let $k \in \{1,\dots,n\}$ and $(S,T)$ with $S \sqcup T = [k-1]$, be a node at level $k-1$. This node's children are $\big(S\sqcup\{k\},T\big)$ and $\big(S,T\sqcup\{k\}\big)$. By \cref{Th:GeneralizedObreshkoff'sTheorem}, it is sufficient to show that for every $0 \leq \alpha \leq 1$, the polynomial
\begin{align*}
\alpha \, q_k \big(S\sqcup\{k\},T\big) + (1-\alpha) \, q_k \big(S,T\sqcup\{k\}\big)
\end{align*}
is real rooted. Denote the above polynomial by $p_\alpha$. By \eqref{Eq:q_k}, we have
\begin{align*}
p_\alpha(x) 
&= \alpha \, q_k \big(S\sqcup\{k\},T\big)(x) + (1-\alpha) \, q_k \big(S,T\sqcup\{k\}\big)(x) \\
&= \Bigg[ \big(\alpha \partial_{z_k} + (1-\alpha) \partial_{y_k}\big) \, \partial_\zv^S \partial_\yv^T \prod_{i=k+1}^n \Big( \partial_{z_i}+\partial_{y_i} \Big) \, g(\zv) g(\yv)  \Bigg] \Bigg|_{\zv=\yv=x\1}.
\end{align*}
By the characterization of stability preserving operators in \cite{Borcea2010Multivariate}, for every $a,b \in \R_{\geq 0}$ and $j \in [n]$, the Weyl operator $a\partial_{z_j}+b\partial_{y_j}$ is real stability preserving. Also, by parts 1 and 3 of \cref{Prop:R.S.P.Properties}, the differential operator $\partial_\zv^S \partial_\yv^T$ and diagonalization of polynomials are real stability preserving. Therefore, since $g(\zv)g(\yv)$ is real stable, $p_\alpha$ is real stable and thus real rooted.
\end{proof}

We proved that the family introduced in \cref{Def:OurInterlacingFamily} is an interlacing family in which the set of leaf-polynomials is $\{g_\Scal : \Scal \in \Pcal_2\}$.  Now, it follows from \cref{Th:InterlacingFamilies} that there exists a $\Scal \in \Pcal_2$ such that $\m ( g_\Scal ) \leq \m ( g_2 )$. This completes the proof of \cref{Th:ExistenceOfGoodPartition}.

\begin{Rem}\label{Rem:Generalr}
For the general $r$, the interlacing family is similar to \cref{Def:OurInterlacingFamily}, with the difference that each node at level $k$ corresponds to a $r$-tuple $(S_1,\dots,S_r)$ with $S_1 \sqcup \dots \sqcup S_r = [k]$, its children are
\begin{align*}
\big(S_1 \cup \{k+1\},\dots,S_r\big),\dots,\big(S_1,\dots,S_r \cup \{k+1\}\big),
\end{align*}
and for each leaf with label $\Scal \in \Pcal_r(n)$, we have $q_n(\Scal) = g_\Scal$. We can compute the attached polynomials similar to \eqref{Eq:q_k}: for a node $(S_1,\dots,S_r)$ at level $k$,
\begin{align*}
q_k(S_1,\dots,S_r)(x) = \Bigg[ \partial_{\zv_1}^{[k] \backslash S_1} \dots \partial_{\zv_r}^{[k] \backslash S_r} \prod_{i=k+1}^n \! \Delta_i \ g(\zv_1) \dots g(\zv_r)  \Bigg] \Bigg|_{\zv_1=\dots=\zv_r=x\1},
\end{align*}
with $\Delta_i := \partial_{\zv_{1i}} \dots \partial_{\zv_{(r-1)i}} + \partial_{\zv_{1i}} \partial_{\zv_{3i}} \dots \partial_{\zv_{ri}} + \dots + \partial_{\zv_{2i}} \dots \partial_{\zv_{ri}}$. Also using the characterization of stability preservers in \cite{Borcea2010Multivariate}, we can prove a result similar to \cref{Lem:OurInterlacingFamilyIsInterlacing}.
\end{Rem}

\subsubsection{The Barrier Method: Upper Bound for \texorpdfstring{$\boldsymbol{\m(g_r)}$}{maxroot(gr)}}\label{Subsec:ProofOfUpperBoundForMaxrootOfg_r}

Now, we prove the inequality \eqref{Eq:UpperBoundForMaxrootOfg_r}. We proceed as in Section 4 of \cite{Leake2020Mixed}.

\begin{Def}
Given a real stable polynomial $p \in \Hcal_n(\R)$ and a point $u \in \R^n$ with $p(u) \neq 0$, the \textit{barrier function in the direction} $i$ at $u$, denoted $\Phi_p^i(u)$, is defined by
\begin{align*}
\Phi_p^i(u) = \dfrac{\partial_ip}{p}(u).
\end{align*}
\end{Def}

\begin{Def}
Let $p \in \R[z_1,\dots,z_n]$. A point $u \in \R^n$ is \textit{above the roots} of $p$ if
\begin{align*}
\forall \, w \geq u \ : \ p(w) \neq 0.
\end{align*}
We will use $\ab_p$ to denote the set of all the points above the roots of $p$.
\end{Def}

The idea behind the barrier method is that the evolution of the above the roots of a real stable polynomial under simple differential operators is governed by the barrier functions. For example, it is known that for every $u \in \ab_p$ and $i \in [n]$, we have $\Phi_p^i(u) > 0$ (for a proof of this fact see \cite{Tao2013Real}). In particular, $\partial_i p(u) \neq 0$, from which it follows that $\ab_p \subseteq \ab_{\partial_i p}$.

Let $g$ satisfy the assumptions of \cref{Th:PolynomialPaving}. Because of the expression give in \cref{Prop:CombinatorialExpressionOfg_r} for $g_r$, to prove \eqref{Eq:UpperBoundForMaxrootOfg_r} it is sufficient to show that
\begin{align*}
\bigg( \sqrt{\dfrac{1}{r} - \dfrac{\alpha}{r-1}} + \sqrt{\alpha} \bigg)^2 \; \1 \in \ab_{\partial_n^{r-1} \dots \partial_1^{r-1} g^r}.
\end{align*}
To prove this, we will begin from a point above the roots of $g^r$ and follow its evolution under iterative applications of operators $\partial_1^{r-1},\dots,\partial_n^{r-1}$ on $g^r$ to obtain a point above the roots of $\partial_n^{r-1} \dots \partial_1^{r-1} g^r$. In each iteration, the point will move back along one of the axis and we will estimate its displacements in terms of the barrier functions. We will also need a control on the behavior of the barrier functions during this process. These are done in \cref{Prop:Delta->Phi} and \cref{Prop:!->delta}.

We will use the following lemma to ensure that a point stays above the roots when we move it along one of the axis.

\begin{lemma}\label{Lem:AbLemma}
Let $p \in \Hcal_n(\R)$, $u \in \ab_p$ and $v \in \R_{\geq0}^n$. If  $p(u-tv) \neq 0$ for all $t \in [0,1]$, then $u-v \in \ab_p$.
\end{lemma}

Our proof for the above lemma relies on results from the theory of hyperbolic polynomials and, to keep the continuity, we have deferred it to \cref{App:HyperbolicPolynomials}. Roughly speaking, the above result holds because above the roots of a real stable polynomial is a convex set.

\begin{Prop}\label{Prop:Delta->Phi}
Let $j \in [n]$ and $p \in \Hcal_n(\R)$ be a real stable polynomial of degree at most $r$ in $z_j$. If $u \in \ab_p$ and $\delta$ satisfies
\begin{align}\label{Eq:Delta->PhiDeltaIneq.}
\Bigg( \partial_i \bigg( \dfrac{\partial_j^{r-1} p}{p} \bigg) (u) \Bigg) \Bigg( \partial_i \bigg( \dfrac{\partial_j^r p}{p} \bigg) (u) \Bigg)^{-1} \leq \delta < \bigg( \Phi_{\partial_j^{r-1}p}^j(u) \bigg)^{-1}
\end{align}
for some $i \in [n]$, then $u - \delta e_j \in \ab_{\partial_j^{r-1} p}$ and
\begin{align}\label{Eq:Delta->PhiPhiIneq.}
\Phi_{\partial_j^{r-1} p}^i (u - \delta e_j) \leq \Phi_p^i (u).
\end{align}
\end{Prop}

\begin{proof}
Since $\mathrm{deg}_j(p) \leq r$, the Taylor expansion of $p$ with respect to $z_j$ is
\begin{align*}
p(u - t e_j) = \sum_{k=0}^r \, (\partial_j^k p)(u) \, \dfrac{(-t)^k}{k!}.
\end{align*}
Therefore,
\begin{align}\label{Eq:Delt->Phi1}
(\partial_j^{r-1}p) (u - t e_j) = \partial_j^{r-1}p(u) - t \partial_j^rp(u).
\end{align}
It follows from the above equation that if $t < \big(\Phi_{\partial_j^{r-1}p}^j(u)\big)^{-1}$, then $(\partial_j^{r-1}p) (u - t e_j) > 0$. Hence assuming $\delta < \big( \Phi_{\partial_j^{r-1}p}^j(u) \big)^{-1}$, we have $(\partial_j^{r-1}p) (u - t \delta e_j) \neq 0$ for every $t \in [0,1]$. Also, note that $u \in \ab_p \subseteq \ab_{\partial_j^{r-1}p}$. Therefore, it follows from \cref{Lem:AbLemma} that $u - \delta e_j \in \ab_{\partial_j^{r-1} p}$.

By \eqref{Eq:Delt->Phi1}, the inequality \eqref{Eq:Delta->PhiPhiIneq.} is equivalent to
\begin{align*}
\dfrac{\partial_i \partial_j^{r-1} p - \delta \partial_i \partial_j^r p}{\partial_j^{r-1} p - \delta \partial_j^r p} (u) \leq \dfrac{\partial_i p}{p} (u).
\end{align*}
By a straightforward calculation, the above inequality is equivalent to the first inequality in \eqref{Eq:Delta->PhiDeltaIneq.}. See \cite[Proposition 4.2]{Leake2020Mixed} for more details.
\end{proof}

The following proposition gives a simpler condition that implies \eqref{Eq:Delta->PhiDeltaIneq.}.

\begin{Prop}\label{Prop:!->delta}
Let $j \in [n]$ and $p \in \Hcal_n(\R)$ be a real stable polynomial of degree at most $r$ in $z_j$. If $u \in \ab_p$ and $\delta$ satisfies
\begin{align}\label{Eq:!->delta}
0 \leq \delta \leq \dfrac{(r-1)^2}{r} \left( \dfrac{1}{\Phi_p^j(u) - \dfrac{1}{u_j - \lambda_r}} \right),
\end{align}
where $\lambda_r$ is the smallest root of the univariate polynomial $p(u_1,\dots,u_{j-1},z_j,u_{j+1},\dots,u_n)$, then
\begin{align}\label{Eq:!->delta-deltaIneq.}
\Bigg( \partial_i \bigg( \dfrac{\partial_j^{r-1} p}{p} \bigg) (u) \Bigg) \Bigg( \partial_i \bigg( \dfrac{\partial_j^r p}{p} \bigg) (u) \Bigg)^{-1} \leq \delta < \bigg( \Phi_{\partial_j^{r-1}p}^j(u) \bigg)^{-1}
\end{align}
for every $i \in [n]$.
\end{Prop}

\begin{proof}
The first inequality in \eqref{Eq:!->delta-deltaIneq.} is proved in \cite[Proposition 4.3]{Leake2020Mixed}. To prove the second inequality, we use inequality (8) from the proof of \cite[Proposition 4.3]{Leake2020Mixed}, which is
\begin{align*}
\delta \leq \sum_{k=1}^{r-1} \dfrac{u_j - \lambda_k}{r},
\end{align*}
where $\lambda_1 \geq \dots \geq \lambda_r$ are the roots of $p(u_1,\dots,u_{j-1},z_j,u_{j+1},\dots,u_n)$. Since $u \in \ab_p$, we have $u_j  > \lambda_k$ for all $k \in [r]$. Therefore,
\begin{align*}
\delta \leq \sum_{k=1}^{r-1} \dfrac{u_j - \lambda_k}{r} \leq \sum_{k=1}^r \dfrac{u_j - \lambda_k}{r} = \dfrac{\partial_j^{r-1}p}{\partial_j^rp} (u) = \bigg( \Phi_{\partial_j^{r-1}p}^j(u) \bigg)^{-1}.
\end{align*}
\end{proof}

We have showed that the evolution of above the roots of a polynomial is related to the barrier functions and the smallest roots of their one dimensional restrictions. Now we estimate these quantities in our problem. We will provide a lower bound for $\lambda_r$ in \cref{Prop:LowerBoundForlambda_r} and an upper bound for $\Phi_p^j(u)$ in \cref{Prop:UpperBoundForPhi}.
%
%

\begin{Prop}\label{Prop:LowerBoundForlambda_r}
Let $g$ and $r$ be as in \cref{Th:PolynomialPaving} and $p := g^r$. Also, let $u \in \ab_g$. Then, for every $(i_1,\dots,i_n) \in \mathbb{Z}^n$ such that $0 \leq i_k \leq r-1$ for all $k \in [n]$, and every $i \in [n]$, all the roots of the univariate polynomial
\begin{align*}
q(z_i) := \bigg[ \bigg( \prod_{k=1}^n \partial_k^{i_k} \bigg) p \bigg] (u_1,\dots,u_{i-1},z_i,u_{i+1},\dots,u_n)
\end{align*}
are non-negative.
\end{Prop}

The above proposition is a generalization of \cite[Proposition 4.6]{Leake2020Mixed} and to prove it, we proceed similar to \cite{Leake2020Mixed}. We need the following two lemmas.

%

\begin{lemma}\label{Lem:LemmaForRootBoundForg}
Let $g$ be as in \cref{Th:PolynomialPaving} and $g = r + z_n s$, where $r,s \in \R[z_1,\dots,z_{n-1}]$. We have $\ab_g \subseteq \ab_s$ and either $r=0$ or $\ab_s \subseteq \ab_r$.
\end{lemma}

\begin{proof}
Note that $s = \partial_1 g$. As we mentioned before, it is known that $\ab_g \subseteq \ab_{\partial_1g}$. Now, assuming $r \neq 0$, we prove $\ab_s \subseteq \ab_r$. Let $u \in \ab_s$ and $w,v \in \R^n$ be such that $v < u \leq w$ and $e := w-v$. By \cref{Prop:ExamplesOfInterlacing} we have $s(te+v) \ll r(te+v)$. Note that $s(te+v)$ is of degree $n-1$ and its leading coefficient is positive. Since $[n] \in \supp(g)$ and $r \neq 0$, it follows from \cref{Cor:SuppPositiveNonnegativeRoots} that the maximum degree monomial of $r$ is $a_{\{n\}} \, \zv^{\{n\}^c}$. Also by \cref{Lem:PositiveRootedPolynomialCoefficientSigns}, we have $a_{\{n\}} < 0$. Therefore, $r(te+v)$ is of degree $n-1$ and its leading coefficient is negative. Hence $\lambda_1\big(r(te+v)\big) \leq \lambda_1\big(s(te+v)\big)$.

Since $u \in \ab_s$, we have $\lambda_1\big(s(te+v)\big) < 1$. Therefore, $\lambda_1\big(r(te+v)\big) < 1$ and hence $r(w) \neq 0$. We showed that $r(w) \neq 0$ for every $w \geq u$. This completes the proof.
\end{proof}

\begin{lemma}\label{Lem:RootBoundForg}
Let $g$ be as in \cref{Th:PolynomialPaving} and $u \in \ab_g$. Then, $g(u_1,\dots,u_{n-1},z_n)$ is a univariate affine polynomial with positive leading coefficient and non-negative root.
\end{lemma}

\begin{proof}
Because $g$ is multi-affine, we can write $g = r + z_n s$, where $r,s \in \R[z_1,\dots,z_{n-1}]$. It follows from \cref{Lem:LemmaForRootBoundForg} that $u \in \ab_s$. Therefore, $s(u) \neq 0$ and hence $g(u_1,\dots,u_{n-1},z_n)$ is a univariate affine polynomial and its root, which we denote by $\lambda$, is $- r(u)/s(u)$.

Note that the sign of a polynomial does not change above its roots. The leading coefficient of $\ovl{s}(x)$ is $1$. Thus $\lim_{x \to \infty} \ovl{s}(x) > 0$ and so we have $s(u) > 0$. 
Also, $u \in \ab_r$  by \cref{Lem:LemmaForRootBoundForg}. As in the proof of the previous lemma, it follows from \cref{Cor:SuppPositiveNonnegativeRoots} that either $r=0$ or the leading coefficient of $\ovl{r}(x)$ is negative. Therefore, $r(u) \leq 0$. Overall, we have $s(u) > 0$ and $r(u) \leq 0$ which implies $\lambda \geq 0$.
\end{proof}

\begin{proof}[Proof of \cref{Prop:LowerBoundForlambda_r}]
Since $g$ is multi-affine,
\begin{align*}
\bigg( \prod_{k=1}^n \partial_k^{i_k} \bigg) p (\zv) = \sum_{(A_1,\dots,A_r) \in \Acal} \ \prod_{i=1}^r \, \partial^{A_i} g (\zv),
\end{align*}
where $\Acal$ is an appropriate subset of $\big(2^{[n]}\big)^r$ whose exact form is not important in this proof. Using an inductive argument similar to the one used in the proof of \cref{Lem:PositiveRootedPolynomialCoefficientSigns}, we can show that the roots $\ovl{\partial^A g}$ lie in the interval $[0,1]$ for $A \subseteq [n]$. It follows that $\partial^A g$ satisfies the assumptions of \cref{Th:PolynomialPaving} for every $A \subseteq [n]$. Therefore, by \cref{Lem:RootBoundForg}, each $\partial^{A_i} g$ is negative at all points $(u_1,\dots,u_{i-1},z_i,u_{i+1},\dots,u_n)$ with $z_i < 0$. Thus the same holds for $p$ and the proposition follows.
\end{proof}

The following proposition which is a generalization of \cite[Lemma 5.3]{Leake2020Mixed}, provides an upper bound for the barrier functions.

\begin{Prop}\label{Prop:UpperBoundForPhi}
Let $g$ and $r$ be as in \cref{Th:PolynomialPaving} and $p := g^r$. If $b \geq 1$, then
\begin{align}
\forall i \in [n] \ : \ \Phi_p^i(b\1) \leq r \bigg( \dfrac{\alpha}{b-1} + \dfrac{1-\alpha}{b} \bigg).
\end{align}
\end{Prop}

We need the following three lemmas.

\begin{lemma}[Lemma 9.B.3 of \cite{Marshall2011Inequalities}]\label{Lem:FromMarshal}
Given real numbers $c_1,\dots,c_{n-1}$ and $\lambda_1,\dots,\lambda_n$ satisfying the interlacing propery
\begin{align*}
\lambda_1 \geq c_1 \geq \lambda_2 \geq \dots \geq c_{n-1} \geq \lambda_n,
\end{align*}
there exists a real symmetric $n \times n$ matrix of the form
\begin{align*}
W =
\begin{bmatrix}
D_c & v^t \\
v & v_n
\end{bmatrix}
\end{align*}
with eigenvalues $\lambda_1,\dots,\lambda_n$.
\end{lemma}

\begin{lemma}[Lemma 5.1 of \cite{Leake2020Mixed}]\label{Lem:FromLeake1}
For any matrix $A \in \Mcal_n(\C)$ and any vector $v \in \C^n$, we have
\begin{align*}
\det(A_{v^\perp}) = (v^*A^{-1}x) \, \det(A),
\end{align*}
where $A_{v^\perp} \in \Mcal_{n-1}(\C)$ is the compression of A onto $v^\perp$.
\end{lemma}

\begin{lemma}[Lemma 5.3 of \cite{Leake2020Mixed}]\label{Lem:FromLeake2}
Let $A \in \Mcal_n(\C)$ be positive semi-definite contraction, $i \in [n]$ and $A_{i,i} \leq \alpha$. Then, for any $a \geq 1$,
\begin{align*}
e_i^* (aI - A)^{-1} e_i \leq \dfrac{\alpha}{a-1} + \dfrac{1-\alpha}{a}.
\end{align*}
\end{lemma}

\cite[Lemma 5.3]{Leake2020Mixed} is slightly weaker than the above lemma but its proof only uses these weaker assumptions. Now, we are ready to prove \cref{Prop:UpperBoundForPhi}.

\begin{proof}[Proof of \cref{Prop:UpperBoundForPhi}]
Let $\gamma_1 \geq \dots \geq \gamma_n$ be the roots of $\ovl{g}$ and $\delta_1 \geq \dots \geq \delta_{n-1}$ be the roots of $\ovl{\partial_ig}$. We have
\begin{align}\label{Eq:UpperBoundForPhi1}
\Phi_p^i(b\1) = \dfrac{r g^{r-1} \partial_i g}{g^r}(b\1) = r \, \dfrac{\ovl{\partial_ig}(b)}{\ovl{g}(b)} = r \, \dfrac{\prod_{i=1}^{n-1} (b - \delta_i)}{\prod_{i=1}^n (b - \gamma_i)}.
\end{align}
Since $\ovl{\partial_ig} \ll \ovl{g}$, we have $\gamma_1 \geq \delta_1 \geq \gamma_2 \geq \dots \geq \delta_{n-1} \geq \gamma_n$. Thus, by \cref{Lem:FromMarshal}, there is a $n \times n$ real symmetric matrix
\begin{align*}
A =
\begin{bmatrix}
D_\delta & v^t \\
v & v_n
\end{bmatrix}
\end{align*}
with $D_\delta = \mathrm{Diag}(\delta_1,\dots,\delta_{n-1})$, whose eigenvalues are $\gamma_1,\dots,\gamma_n$. By \cref{Lem:FromLeake1},
\begin{align}\label{Eq:UpperBoundForPhi2}
\dfrac{\prod_i^{n-1} (b - \delta_i)}{\prod_i^n (b - \gamma_i)} = \dfrac{\det(bI_{n-1} - D_\delta)}{\det(bI_n - A)} = e_n^t (bI_n - A)^{-1} e_n,
\end{align}
where $I_k$ denotes the $k \times k$ identity matrix. We have $v_n = \sum_{i=1}^n \gamma_i - \sum_{i=1}^{n-1} \delta_i = a_{\{i\}} \leq \alpha$. Also by the assumption, $\gamma_i \in [0,1]$ for all $i \in [n]$ and thus $A$ is a positive semi-definite contraction. Therefore, by \cref{Lem:FromLeake2},
\begin{align}\label{Eq:UpperBoundForPhi3}
e_n^t (bI_n - A)^{-1} e_n \leq \dfrac{\alpha}{b-1} + \dfrac{1-\alpha}{b}.
\end{align}
The lemma follows from \eqref{Eq:UpperBoundForPhi1}, \eqref{Eq:UpperBoundForPhi2} and \eqref{Eq:UpperBoundForPhi3}.
\end{proof}

The following lemma provides a subset of $\ab_g$ which will serve as the set of starting points for our iterative argument over which we will then optimize.

\begin{lemma}\label{Lem:AboveRootsOfg}
Assume that $g \in \R[z_1,\dots,z_n]$ is a multi-affine real stable polynomial such that $[z_1 \dots z_n]_g = 1$ and all the roots of $\ovl{g}$ are in $[0,1]$. Then, for every $b>1$, the point $b\1 \in \R^n$ is above the roots of $g$.
\end{lemma}

\begin{proof}
We proceed by induction on $n$. The case $n=1$ is obvious. Suppose that the statement is true for $n-1$. Consider the polynomial $\partial_n g \in \R[z_1,\dots,z_{n-1}]$. Note that $[z_1 \dots z_{n-1}]_{\partial_n g} = 1$. Also, since $\ovl{\partial_n g} \ll \ovl{g}$ and $\mathrm{deg}(\ovl{\partial_n g}) = \mathrm{deg}(\ovl{g}) - 1$, all the roots of $\ovl{\partial_ng}$ are in the interval $[0,1]$. Therefore, by the induction hypothesis, $b\1$ is above the roots of $\partial_n g$.

Since $\lambda_i(\ovl{\partial_n g}) \leq 1$ for $i=1,\dots,n-1$, we have $\ovl{\partial_n g}(b) > 0$. Therefore, since the sign of a polynomial does not change above its roots, $\partial_n g$ is positive above $b\1$. Hence $g$ is increasing in $z_n$ above $b\1$. The same argument works for all the other variables. Therefore, for every $w \in \R^n$ with $w \geq b\1$, we have
\begin{align*}
g(w) \geq g(b\1) = \ovl{g}(b) > 0.
\end{align*}
The last inequality holds since $\lambda_i(\ovl{g}) \leq 1$ for $i=1,\dots,n$. Thus $g(w) \neq 0$ for every $w \geq b\1$ which means $b\1 \in \ab_g$.
\end{proof}

\begin{lemma}[Lemma 5.5 of \cite{Leake2020Mixed}]\label{Lem:Optimization}
For $\alpha,\beta \in [0,1]$, we have
\begin{align*}
\inf_{a>1} \left( a - \dfrac{\beta}{\dfrac{\alpha}{a-1} + \dfrac{1-\alpha}{a}} \right) =
\begin{dcases}
\big( \sqrt{\alpha\beta} + \sqrt{(1-\alpha)(1-\beta)} \, \big)^2 & , \alpha \leq \beta \\
1 & , \alpha \geq \beta
\end{dcases}.
\end{align*}
\end{lemma}

We have generalized all the tools that are used in \cite{Leake2020Mixed} and so the proof of \cite[Theorem 5.6]{Leake2020Mixed} works for \eqref{Eq:UpperBoundForMaxrootOfg_r}. For the sake of completeness we repeat the argument here.

\begin{Th}
Let $g_r$ be as in \cref{Prop:CombinatorialExpressionOfg_r}. We have
\begin{align*}
\m(g_r) \leq \bigg( \sqrt{\dfrac{1}{r} - \dfrac{\alpha}{r-1}} + \sqrt{\alpha} \bigg)^2.
\end{align*}
\end{Th}

\begin{proof}
Fix $b>1$ and define $w_0 \in \R^n$ as $w_0 = b\1$. By \cref{Lem:AboveRootsOfg}, $w_0$ is above the roots of $g^r$. Let $p_0 = g^r$ and iteratively define
\begin{align*}
p_k = \partial_k^{r-1} p_{k-1}, \quad k=1,\dots,n,
\end{align*}
and
\begin{align*}
\delta_k = \dfrac{(r-1)^2}{r} \left( \dfrac{1}{\Phi_{p_{k-1}}^k\big(w_{k-1}\big) - \dfrac{1}{b}} \right) \qquad \text{and} \qquad w_k = w_{k-1} - \delta_k e_k.
\end{align*}
By \cref{Prop:Delta->Phi}, \cref{Prop:!->delta} and \cref{Prop:LowerBoundForlambda_r}, for every $k \in [n]$ we have $w_k \in \ab_{p_k}$ and
\begin{align*}
\forall i \in [n] \ : \ \Phi_{p_k}^i \big(w_{k-1} - \delta_k e_k\big) \leq \Phi_{p_{k-1}}^i \big(w_{k-1}\big).
\end{align*}
It follows from \cref{Prop:UpperBoundForPhi} that
\begin{align*}
\delta_k \geq \dfrac{(r-1)^2}{r} \left( \dfrac{1}{\Phi_p^k(b\1) - \dfrac{1}{b}} \right) \geq \dfrac{(r-1)^2}{r} \left( \dfrac{1}{r \bigg( \dfrac{\alpha}{b-1} + \dfrac{1-\alpha}{b} \bigg) - \dfrac{1}{b}} \right) =: \delta.
\end{align*}
Hence $(b-\delta)\1 \geq w_n$ and since $w_n \in \ab_{p_n}$, we have $(b-\delta)\1 \in \ab_{p_n}$. Also, by \cref{Prop:CombinatorialExpressionOfg_r}, we have $g_r = \ovl{p}_n$. Therefore,
\begin{align*}
\m(g_r)
&\leq \inf_{b>1} \left\{ b - \dfrac{(r-1)^2}{r} \left( \dfrac{1}{r \bigg( \dfrac{\alpha}{b-1} + \dfrac{1-\alpha}{b} \bigg) - \dfrac{1}{b}} \right) \right\} \\
&= \inf_{b>1} \left\{ b - \dfrac{r-1}{r} \left( \dfrac{1}{\dfrac{r\alpha/(r-1)}{b-1} + \dfrac{1-r\alpha/(r-1)}{b}} \right) \right\}.
\end{align*}
By \cref{Lem:Optimization}, if $(r-1)^2/r^2 \geq \alpha$ then
\begin{align*}
\m(g_r) \leq \bigg( \sqrt{\dfrac{1}{r} - \dfrac{\alpha}{r-1}} + \sqrt{\alpha} \bigg)^2.
\end{align*}
\end{proof}

\subsection{Paving Property for Polynomials, Second Version}\label{Subsec:ProofOfZeroDiagonalPolynomialPaving}

The following proposition is an extension of \cite[Corollary 26]{Tao2013Real}.

\begin{Prop}\label{Prop:ZeroDiagonalPolynomialPaving}
Let $r \geq 4$ be an integer and $\Lambda \in R_+$. Assume that $g \in \R[z_1,\dots,z_n]$ is a multi-affine real stable polynomial and $g(\zv) = \sum_{A \subseteq [n]} a_A \, \zv^{A^c}$. If all the roots of $\ovl{g}$ are in $[-\Lambda,\Lambda]$, $a_\emptyset=1$ and $a_{\{i\}}=0$ for $i = 1,\dots,n$, then there exists a partition $S_1,\dots,S_{r^2}$ of $[n]$ such that
\begin{align*}
\forall i \in [r^2] \ : \ \M \big( \, \ovl{\partial^{S_i^c}g} \, \big) \leq \left( \dfrac{r-2}{r(r-1)} + 2 \, \sqrt{\dfrac{r-2}{r(r-1)}} \, \right) \Lambda,
\end{align*}
where $\M(.)$ is the maximum absolute value of roots.
\end{Prop}

We proceed as in \cite{Tao2013Real}. The following two lemmas are generalizations of \cite[Corollary 24]{Tao2013Real} and \cite[Corollary 25]{Tao2013Real}.

\begin{lemma}\label{Lem:ZeroDiagonalPolynomialPavingStep1}
Let $r \in \Z$ and $\alpha,\Lambda \in R_+$ be such that $r \geq 2$ and $\alpha \leq \Lambda (r-1)^2/r^2$. Assume that $g \in \R[z_1,\dots,z_n]$ is a multi-affine real stable polynomial and $g(\zv) = \sum_{A \subseteq [n]} a_A \, \zv^{A^c}$. If all the roots of $\ovl{g}$ are in the interval $[0,\Lambda]$, $a_\emptyset=1$ and $|a_{\{i\}}| \leq \alpha$ for $i = 1,\dots,n$, then there exists a partition $S_1,\dots,S_r$ of $[n]$ such that
\begin{align*}
\forall i \in [r] \ : \ \m \big( \, \ovl{\partial^{S_i^c}g} \, \big) \leq \bigg( \sqrt{\dfrac{\Lambda}{r} - \dfrac{\alpha}{r-1}} + \sqrt{\alpha} \bigg)^2.
\end{align*}
\end{lemma}

\begin{proof}
Define $f(\zv) = \Lambda^{-n} \, g(\Lambda \cdot \zv)$ and apply \cref{Th:PolynomialPaving} to the polynomial $f$
\end{proof}


\begin{lemma}\label{Lem:ZeroDiagonalPolynomialPavingStep2}
Let $r \in \Z$ and $\Lambda,\Gamma \in \R_+$ be such that $r \geq 2$ and $(r-1)^2/r^2 \geq \Lambda/(\Lambda+\Gamma)$. Assume that $g \in \R[z_1,\dots,z_n]$ is a multi-affine real stable polynomial and $g(\zv) = \sum_{A \subseteq [n]} a_A \, \zv^{A^c}$. If all the roots of $\ovl{g}$ are in $[-\Lambda,\Gamma]$, $a_\emptyset=1$ and $a_{\{i\}} = 0$ for $i = 1,\dots,n$, then there exists a partition $S_1,\dots,S_r$ of $[n]$ such that
\begin{align*}
\forall i \in [r] \ : \ -\Lambda \cdot \1 \leq \lambda \big( \, \ovl{\partial^{S_i^c}g} \, \big) \leq \Bigg[ \bigg( \sqrt{\dfrac{\Lambda+\Gamma}{r} - \dfrac{\Lambda}{r-1}} + \sqrt{\Lambda} \bigg)^2 \! - \Lambda \Bigg] \cdot \1,
\end{align*}
where $\lambda(\cdot)$ denotes the vector of the roots of a polynomial in the non-increasing order.
\end{lemma}

\begin{proof}
Apply \cref{Lem:ZeroDiagonalPolynomialPavingStep1} with $g(\zv)$ replaced by $g(\zv - \Lambda\cdot\1)$, $\Lambda$ replaced by $\Lambda+\Gamma$, and $\alpha$ and $\varepsilon$ both set equal to $\Lambda$.
\end{proof}

Now we are ready to prove \cref{Prop:ZeroDiagonalPolynomialPaving}.

\begin{proof}[Proof of \cref{Prop:ZeroDiagonalPolynomialPaving}]
Set
\begin{align*}
c := \bigg( \sqrt{\dfrac{2\Lambda}{r} - \dfrac{\Lambda}{r-1}} + \sqrt{\Lambda} \bigg)^2 \! - \Lambda = \left( \dfrac{r-2}{r(r-1)} + 2 \, \sqrt{\dfrac{r-2}{r(r-1)}} \, \right) \Lambda.
\end{align*}
The roots of $\ovl{g}$ lie between $-\Lambda$ and $\Lambda$. Hence, by \cref{Lem:ZeroDiagonalPolynomialPavingStep2}, there is a partition $S_1,\dots,S_r$ of $[n]$ such that the roots of each $\ovl{\partial^{S_i^c} g}$ lie between $-\Lambda$ and $c$. Note that $c \leq \Lambda$ and the roots of each of these polynomials are in $[-\Lambda,\Lambda]$. For each $i \in [r]$, define the polynomial $f_i$ as $f_i(\zv) = (-1)^n \partial^{S_i^c} g(-\zv)$. Each $f_i$ satisfies the assumptions of \cref{Lem:ZeroDiagonalPolynomialPavingStep2} with $\Gamma=\Lambda$ (the roots of $\ovl{f}_i$ are in $[-\Lambda,\Lambda]$ since its roots are negations of the roots of $\ovl{\partial^{S_i^c} g}$). Therefore, for every $i \in [r]$, there is a partition $S_{i,1},\dots,S_{i,r}$ of $S_i$ such that the roots of each $\ovl{\partial^{S_{i,j}^c} f_i}$ lie between $-\Lambda$ and $c$.

By regarding each $S_{i,j}$ as a subset of $[n]$, we have $\partial^{S_{i,j}^c} f_i(\zv) = (-1)^n \partial^{S_{i,j}^c} g(-\zv)$. Thus the roots of $\ovl{\partial^{S_{i,j}^c} g}$ are negations of the roots of $\ovl{\partial^{S_{i,j}^c} f_i}$ and hence the roots of each $\ovl{\partial^{S_{i,j}^c} g}$ lie between $-c$ and $\Lambda$. Therefore,
\begin{align*}
\lambda_\sharp\big(\,\ovl{\partial^{S_{i,j}^c} g}\,\big) \geq -c,
\end{align*}
where $\lambda_\sharp(p)$ denotes the least root of polynomial $p$. Also, by interlacing,
\begin{align*}
\lambda_1 \big(\,\ovl{\partial^{S_{i,j}^c} g}\,\big) \leq \lambda_1 \big(\,\ovl{\partial^{S_i^c} g}\,\big) \leq c.
\end{align*}
Therefore,
\begin{align*}
\M \big( \, \ovl{\partial^{S_{i,j}^c}g} \, \big) \leq c = \left( \dfrac{r-2}{r(r-1)} + 2 \, \sqrt{\dfrac{r-2}{r(r-1)}} \, \right) \Lambda
\end{align*}
and $\big( S_{i,j} \big)_{i,j \in [r]}$ is the desired partition.
\end{proof}

\section{Paving Property for Strongly Rayleigh Processes}\label{Sec:ProofOfS.R.P.PavingEntropyVersion}

We will prove \cref{Th:S.R.P.PavingEntropyVersion} in \cref{Subsec:StronglyRayleighPaving}. We will introduce the notion of kernel polynomial of strongly Rayleigh processes in \cref{Subsec:S.R.Processes}. This notion will be essential in the proof of \cref{Th:S.R.P.PavingEntropyVersion}. In \cref{Subsec:EntropyLowerBound}, we will prove an entropy lower bound for strongly Rayleigh processes in terms of the roots of the kernel polynomial. This entropy bound will be used in the proof of \cref{Th:S.R.P.PavingEntropyVersion}.

The connection between the probabilistic paving property and the paving conjecture is more apparent in ``determinantal point processes". These point processes are a very well studied class of strongly Rayleigh processes. For a background on detereminantal processes see \cite{Hough2009Zeros}.

A \textit{point process} $\X$ on a finite set $S$ is a random subset of $S$. Note that the law of $\X$ is a probability measure on the lattice of all the subsets of $S$. Alternatively, $\X$ can be identified with its indicator (random) vector, namely $(X_i)_{i \in S}$, where $X_i$ is the indicator function of the event $\{i \in \X\}$. When $|S|=n$ we can replace $S$ with $[n]$ without loss of generality.


A point process $\X$ on $[n]$ is \textit{determinantal} if there exists a Hermitian matrix $K \in \Mcal_n(\C)$, called \textit{kernel} of $\X$, such that for every $A \subseteq [n]$ we have
\begin{align*}
\prb(A \subseteq \X) = \det K_A,
\end{align*}
where $K_A$ is the principal submatrix of $K$ with rows and columns in $A$. It is well-known that a Hermitian matrix $K$ is the kernel of a determinantal process if and only if it is a positive semi-definite contraction (see, e.g., \cite[Theorem 4.5.5]{Hough2009Zeros}). It is proved in \cite{Borcea2009Negative} that determinantal processes have the strong Rayleigh property. We will explain the connection between the kernel polynomial of strongly Rayleigh processes and the kernel of determinantal processes in the next subsection.

In order to obtain the paving property for a discrete determinantal process, we will apply the following version of matrix paving to its kernel.

\begin{Prop}[Corollary 26 of \cite{Tao2013Real}]\label{Prop:MSSZeroDiagonalPaving}
Let $\Lambda$ be a positive number and $r$ be an integer such that $r \geq 2$. For every Hermitian matrix $A \in \Mcal_n(\C)$ with vanishing diagonal and $\|A\|_{op} \leq \Lambda$, there are diagonal projections $P_1,\dots,P_{r^2} \in \Mcal_n(\C)$ such that $\sum_{i=1}^{r^2} P_i = I_n$ and
\begin{align*}
\forall i \in \big[ r^2 \big] \ : \ \| P_i A P_i \|_{op} \leq \bigg( \dfrac{2\sqrt{2}}{\sqrt{r}} + \dfrac{1}{r} \bigg) \Lambda.
\end{align*}
\end{Prop}

Let $K$ be the kernel of a determinantal process $\X$ and $D := \mathrm{Diag}(K)$. By applying the above proposition to $K-D$, we conclude that for every positive $\varepsilon$, there is a positive integer $r$ such that it is possible to partition $[n]$ into $r$ subsets $S_1,\dots,S_r$ such that
\begin{align*}
\big\| K_{S_i} - D_{S_i} \big\|_{op} \leq \varepsilon.
\end{align*}
Note that for every $A \subseteq [n]$, the matrix $K_A$ is the kernel of the restriction of $\X$ to $A$, namely $\X \cap A$, and that a determinantal process has independent points if and only if its kernel is diagonal. Hence, we can interpret the above inequality as the restrictions of the determinantal process to each $S_i$ having ``almost independent points".

In order to extend the above argument to strongly Rayleigh processes, we will apply \cref{Prop:ZeroDiagonalPolynomialPaving} to to the ``kernel polynomial" of a strongly Rayleigh process which will be introduced in the next subsection. As mentioned in \cref{Subsec:ProofOfZeroDiagonalPolynomialPaving}, \cref{Prop:ZeroDiagonalPolynomialPaving} is a generalization of \cref{Prop:MSSZeroDiagonalPaving}. Finally, in order to obtain \cref{Th:S.R.P.PavingEntropyVersion}, we will need a relationship between the entropy of a strongly Rayleigh process and the entropy of the roots of its kernel. This is done in \cref{Subsec:EntropyLowerBound}.

\subsection{Strongly Rayleigh Processes}\label{Subsec:S.R.Processes}

In this subsection we will introduce the notion of kernel polynomial of strongly Rayleigh processes. The kernel polynomial is, in a sense, a generalization of the kernel of determinantal process, hence the name. We will show that the kernel polynomial satisfies most of the known properties of the kernel of determinantal processes.

\begin{Def}
A point process $\X$ on $[n]$ is \textit{strongly Rayleigh} if its probability generating polynomial, defined as
\begin{align*}
f_\X(\zv) = \sum_{A \subseteq [n]} \prb (\X = A) \, \zv ^A,
\end{align*}
is real stable.
\end{Def}

Strongly Rayleigh processes have many nice properties, some of which are as follows:

\begin{enumerate}
\item
Strongly Rayleigh processes have the negative association property. This is proved in \cite{Borcea2009Negative}.
\item
The class of strongly Rayleigh processes is closed under many natural operations, including products, projections, external fields, conditioning and symmetric homogenization. These properties are proved in \cite{Borcea2009Negative}.
\item
Strongly Rayleigh processes have strong concentration properties, e.g., it is proved in \cite{Pemantle2014Concentration} that Lipschitz functionals of strongly Rayleigh processes satisfy an Azuma-type concentration inequality.
\end{enumerate}

For more information on strongly Rayleigh processes see \cite{Borcea2009Negative}.

\begin{Def}\label{Def:Kernel}
For a point process $\X$ on $[n]$, we define its \textit{kernel}, denoted $g_\X$, by
\begin{align*}
g_\X(z_1,\dots,z_n) = z_1 \dots z_n \ f_\X \Big(1-\dfrac{1}{z_1},\dots,1-\dfrac{1}{z_n}\Big),
\end{align*}
where $f_\X$ is the probability generating polynomial of $\X$.
\end{Def}

By computing the coefficients of $g_\X$, we get
\begin{align}\label{Eq:g_XCoefficients}
g_\X(\zv) = \sum_{A \subseteq [n]} (-1)^{|A|} \, \prb(A \subseteq \X) \, \zv^{A^c}.
\end{align}

It is shown in \cite[Proposition 3.5]{Borcea2009Negative} that determinantal processes are strongly Rayleigh and if $\Y$ is a determinantal process with kernel $K$, then $f_\Y(z) = \det(KZ+I-K)$, where $Z = \mathrm{Diag}(z_1,\dots,z_n)$. By a straightforward calculation we get $g_\Y(z) = \det(Z-K)$, namely the kernel polynomial of a determinantal process is the multivariate characteristic polynomial of its kernel.

\begin{Prop}\label{Prop:RealStabilityOfKernel}
The kernel of a strongly Rayleigh point process is real stable.
\end{Prop}

\begin{proof}
Define
\begin{align*}
\mathcal{T}(p)(z_1,\dots,z_n) &= p(1-z_1,\dots,1-z_n), \\
\mathcal{R}(p)(z_1,\dots,z_n) &= z_1 \dots z_n \, p(z_1^{-1},\dots,z_n^{-1}).
\end{align*}
Note that $\mathcal{T}$ is real stability preserving. Also, $\mathcal{R}$ is real stability preserving by part 4 of \cref{Prop:R.S.P.Properties}. The proposition follows since $g_\X = \mathcal{R}(\mathcal{T}(f_\X))$.
\end{proof}

\begin{Rem}\label{Rem:TransformationOfg_XTof_X}
Note that $f_\X = \mathcal{T}(\mathcal{R}(g_\X))$. Therefore, the distribution of a point process is uniquely determined by its kernel.
\end{Rem}

For a determinantal process $\Y$ with kernel $K$ and every $A \subseteq [n]$, the matrix $K_A$ is the kernel of the restriction of $\Y$ to $A$, namely $\Y \cap A$. The following theorem is a generalization of this fact.

\begin{Prop}\label{Prop:RestrictionOfS.R.P.}
Let $\X$ be strongly Rayleigh process on $[n]$ with kernel $g_\X$. For each $A \subseteq [n]$, the polynomial $\partial^{A}g_\X$ is the kernel of the restriction of $\X$ to $A^c$, namely the point process $\X \cap A^c$.
\end{Prop}

\begin{proof}
Recall that $[\zv^\nu]_p$ denotes the coefficient of $\zv^\nu$ in polynomial $p$. We have
\begin{align*}
\big[\zv^{B^c}\big]_{\partial^Ag_\X} =
\begin{dcases*}
(-1)^{|B|} \, \prb(B \subseteq \X) & , if $B \subseteq A^c$ \\
0 & , otherwise
\end{dcases*}.
\end{align*}
Also, note that for every $B \subseteq A^c$ we have $\prb(B \subseteq \X \cap A^c) = \prb(B \subseteq \X)$. The result follows from these two facts.
\end{proof}

Theorem 4.5.5 of \cite{Hough2009Zeros} states that a Hermitian matrix $K$ is the kernel of a determinantal process if and only if it is a positive semi-definite contraction. The following proposition extends this result to strongly Rayleigh processes.

\begin{Th}\label{Th:ClassificationOfKernel}
Let $g \in \R[z_1,\dots,z_n]$ be a multi-affine real stable polynomial. Then $g$ is the kernel of a strongly Rayleigh process if and only if $[z_1 \dots z_n]_g = 1$ and all the roots of $\ovl{g}$ are in the interval $[0,1]$.
\end{Th}

\begin{proof}
First we prove the ``only if" part. Let $\X$ be a strongly Rayleigh process with kernel $g_\X$. It follows from \eqref{Eq:g_XCoefficients} that $[z_1 \dots z_n]_{g_\X} = 1$. Assume that $\lambda_1,\dots,\lambda_n$ are the roots of $\ovl{g}_\X$. We have
\begin{align*}
\ovl{g}_\X(x) = (x-\lambda_1)\dots(x-\lambda_n).
\end{align*}
It follows from the definition of kernel that
\begin{align}\label{Eq:FactorizationOff_X}
\ovl{f}_\X(x) = (1-x)^n \, \ovl{g}_\X \bigg(\dfrac{1}{1-x}\bigg) = (\lambda_1x+1-\lambda_1)\dots(\lambda_nx+1-\lambda_n).
\end{align}
Since the coefficients of $\ovl{f}_\X$ are non-negative, we have $(1-\lambda_i)/\lambda_i \geq 0$ for each non-zero $\lambda_i$. This implies that $\lambda_i \in [0,1]$.

Now, consider the ``if" part. Let
\begin{align*}
g(\zv) = \sum_{A \subseteq [n]} (-1)^{|A|} a_A \, \zv^{A^c}
\end{align*}
By \cref{Lem:PositiveRootedPolynomialCoefficientSigns}, we have $a_A \geq 0$ for all $A \subseteq [n]$. Let $\mathcal{R}$ and $\mathcal{T}$ be as defined in the proof of \cref{Prop:RealStabilityOfKernel} and define $f = \mathcal{T}(\mathcal{R}(g))$. By computing the coefficients,
\begin{align*}
f(\zv) = \sum_{B \subseteq [n]} b_B \, \zv^B, \qquad b_B = \sum_{A \supseteq B} (-1)^{|A \backslash B|} a_A.
\end{align*}
Note that $g$ is the kernel of a strongly Rayleigh process if and only if $f$ is a real stable probability generating polynomial. Since $\mathcal{R}$ and $\mathcal{T}$ are real stability preserving, $f$ is real stable. It remains to prove that $b_B \in [0,1]$ and $\sum_{B \subseteq [n]} b_B = 1$.

For $I \subseteq [n]$ define $g_I = g|_{z_i = 0 \; \text{for} \; i \in I}$. Let $B \subseteq [n]$ and $B = \{i_1,\dots,i_k\}$. Also define $B_l = \{i_1,\dots,i_l\}$ for $l=1,\dots,k$ and $B_0 = \emptyset$. By \cref{Prop:ExamplesOfInterlacing},
\begin{align*}
\ovl{\partial_{i_{l+1}}g}_{B_l} \ll \ovl{g}_{B_l} \;\quad \text{and} \;\quad \ovl{\partial_{i_{l+1}}g}_{B_l} \ll \ovl{g}_{B_{l+1}}
\end{align*}
for $l=0,\dots,k-1$. Assume $\ovl{g}_B \neq 0$. Hence $g_B \neq 0$. Since $[n] \in \supp(g)$ and $g_B \neq 0$, it follows from \cref{Cor:SuppPositiveNonnegativeRoots} that $B^c \in \supp(g)$. Again using \cref{Cor:SuppPositiveNonnegativeRoots}, we have $B_l^c \in \supp(g)$ for $l=1,\dots,k-1$. Therefore, by \cref{Lem:PositiveRootedPolynomialCoefficientSigns}, the leading coefficients of $\ovl{\partial_{i_{l+1}}g}_{B_l}$ and $\ovl{g}_{B_l}$ have the same sign while the leading coefficients of $\ovl{g}_{B_{l+1}}$ and $\ovl{\partial_{i_{l+1}}g}_{B_l}$ have different signs. This implies that $\lambda_1 \big(\ovl{g}_{B_{l+1}}\big) \leq \lambda_1 \big(\ovl{\partial_{i_{l+1}}g}_{B_l}\big) \leq \lambda_1 \big(\ovl{g}_{B_l}\big)$ for $l=0,\dots,k-1$. Therefore,
\begin{align*}
1 \geq \lambda_1(\ovl{g}) \geq \lambda_1\big(\ovl{g}_{\{i_1\}}\big) \geq \lambda_1\big(\ovl{g}_{\{i_1,i_2\}}\big) \geq \dots \geq \lambda_1(\ovl{g}_B).
\end{align*}
Consequently, either $\ovl{g}_B = 0$ or its leading coefficient is $(-1)^{|B|} \, a_B$ and its roots are less than or equal to 1. Hence $(-1)^{|B|} \, \ovl{g}_B(1) \geq 0$. On the other hand,
\begin{align*}
\ovl{g_B}(1) = g_B(\1) = \sum_{A^c \subseteq B^c} (-1)^{|A|} \, a_A = \sum_{A \supseteq B} (-1)^{|A|} \, a_A.
\end{align*}
Therefore,
\begin{align*}
b_B = \sum_{A \supseteq B} (-1)^{|A \backslash B|} a_A = (-1)^{|B|} \, \ovl{g}_B(1) \geq 0.
\end{align*}

Also we have
\begin{align*}
\sum_{B \subseteq [n]} b_B = \sum_{B \subseteq [n]} \sum_{A \supseteq B} (-1)^{|A \backslash B|} a_A = \sum_{A \subseteq [n]} \Bigg(\sum_{B \subseteq A} (-1)^{|A \backslash B|}\Bigg) a_A = a_\emptyset = 1,
\end{align*}
where we used the fact that $\sum_{B \subseteq A} (-1)^{|A \backslash B|} = 1$ when $A = \emptyset$ and it is zero otherwise.
\end{proof}

By comparing the coefficients of the two sides of \eqref{Eq:FactorizationOff_X} we obtain the following result.

\begin{Prop}\label{Prop:S.R.PSizeDistribution}
Let $\X$ be a strongly Rayleigh process on a set of size $n$ and $\lambda_1,\dots,\lambda_n$ be the roots of $\ovl{g}_\X$. By \cref{Th:ClassificationOfKernel}, $\lambda_i \in [0,1]$. Let $I_1,\dots,I_n$ be independent Bernoulli variables with $I_i \sim \mathrm{Bernoulli}(\lambda_i)$. We have
\begin{align*}
|\X| \sim I_1 + \dots + I_n.
\end{align*}
\end{Prop}

It is proved in \cite[Lemma 4.1]{Pemantle2014Concentration} that the size of a strongly Rayleigh process has the same distribution as the sum of independent Bernoulli variables. The above proposition describes this distribution in a canonical way and generalizes a result about determinantal processes. The special case of the above proposition for determinantal processes is proved in \cite[Theorem 4.5.3]{Hough2009Zeros}. In fact, \cite[Theorem 4.5.3]{Hough2009Zeros} is much stronger and provides a canonical probabilistic interpretation for the eigenvalues of the kernel of determinantal processes. A natural question is whether this result extends to strongly Rayleigh processes. We believe that such an extension requires a deeper understanding of the structure of real stable polynomials. We will propose a conjecture in \cref{Subsec:EntropyLowerBound} which can be regarded as a first step in this direction.

\subsection{On The Correlation Structure of Negatively Dependent Measures}\label{Subsec:RepellingProperty}

Positive and negative dependence model attraction and repulsion, respectively. Unlike the positive dependent case, we expect that the correlation structure of a negatively dependent measure is constrained, in the sense that the repulsive force between the points cannot be strong everywhere. This distinction is already apparent in the definitions of positive and negative association, as the negative association property is more restrictive.

Recall that a point process on $[n]$ with law $\mu$, is \textit{positively associated} if
\begin{align}\label{Eq:PADefinition}
\int fg \, d\mu \geq \int f \, d\mu \int g \, d\mu
\end{align}
for all increasing functions $f$ and $g$ on the lattice of all the subsets of $[n]$, while it is \textit{negatively associated} if the reverse inequality holds, but for those increasing functions $f$ and $g$ which depend on disjoint subsets of the $n$ variables (where each subset is identified with its indicator vector). This distinction stems from the fact that a random variable is always positively correlated with itself and consequently, when $f$ and $g$ depend on a common variable, this gives rise to some positive ``auto-correlation" between $f$ and $g$ which works against the negative ``inter-correlations". In the extreme case, because of the Cauchy-Schwarz inequality, the reverse of \eqref{Eq:PADefinition} cannot hold for $f=g$. In this subsection, we will present three other manifestations of the aforementioned restriction in the correlation structure of strongly Rayleigh point processes.

Our first example is  \cite[Lemma 3.2]{Ghosh2017Multivariate}. This result states that for every strongly Rayleigh point process $\X$ with $\X=(X_1,\dots,X_n)$, we have
\begin{align*}
\mathrm{var}(X_i) + \! \sum_{j \in [n] \, , \, j \neq i} \! \mathrm{cov}(X_i,X_j) \geq 0 \, , \quad i=1,\dots,n.
\end{align*}
Since strongly Rayleigh processes have negative pairwise correlations, the above inequality implies that the pairwise correlations must typically be much smaller than the ``variances". There is no such restriction in the positively associated case. For example, in the the extreme case where $X_1=\dots=X_n$ we have $\mathrm{cov}(X_i,X_j) = \mathrm{var}(X_i)$ for all $i,j \in [n]$.

%
%

Our second example is the following theorem which appears in \cite[Corollary 5.6]{Anari2018LogConcave}.

\begin{Th}\label{Th:FromOveisGharan}
Let $\X$ be a strongly Rayleigh process and $\X = (X_1,\dots,X_n)$. We have
\begin{align*}
\dfrac{1}{2} \, \sum_{i=1}^n H(X_i) \leq H(\X),
\end{align*}
where $H(\cdot)$ is the entropy function.
\end{Th}

Recall that for every point process $\X$ with $\X = (X_1,\dots,X_n)$, we have $H(\X) \leq \sum_{i=1}^n H(X_i)$ and equality occurs if and only if $X_1,\dots,X_n$ are independent. The above theorem implies that the correlation structure of a strongly Rayleigh process cannot be very strong, in the sense that its entropy cannot be much smaller than its independent version. On the other hand, the entropy of a positively dependent measure can be significantly smaller than its independent version. For example, if $X_1=\dots=X_n$, then $H(\X) = H(X_i)$, which can be significantly smaller than $\sum_{i}^n H(X_i) = n H(X_i)$.

The paving property for strongly Rayleigh processes (\cref{Th:S.R.P.PavingEntropyVersion}) is also a manifestation of this phenomenon. This theorem states that for a strongly Rayleigh process, the underlying space can be partitioned into a small number of sets such that the points of the restrictions of the process to each set are almost independent. On the other hand, in the positively dependent case all the points can be strongly correlated; for example consider the case $X_1=\dots=X_n$.

\subsection{An Entropy Lower Bound}\label{Subsec:EntropyLowerBound}

Let $\X$ be a strongly Rayleigh process on $[n]$. Recall the Bernoulli variables $I_1,\dots,I_n$ from \cref{Prop:S.R.PSizeDistribution} for which $|\X| \sim I_1+\dots+I_n$. In this subsection we prove that the entropy of $\X$ is greater than or equal to the entropy of $(I_1,\dots,I_n)$. An obvious lower bound for the entropy of $\X$ is $H(|\X|)$. Our result provides a stronger lower bound.

\begin{Th}\label{Th:EntropyLowerBound}
Let $\X$ be a strongly Rayleigh process on $[n]$ with kernel $g_\X$. We have
\begin{align*}
H(\X) \geq \sum_{i=1}^n h(\lambda_i),
\end{align*}
where $\lambda_1,\dots,\lambda_n$ are the roots of $\ovl{g}_\X$.
\end{Th}

We will use the following three lemmas.

\begin{lemma}[5.B.4 of \cite{Marshall2011Inequalities}]\label{Lem:FromMarshall2}
If $b_1 \geq \dots \geq b_{n-1}$ interlaces $a_1 \geq \dots \geq a_n$, then
\begin{align*}
(a_1,\dots,a_n) \succ (b_1,\dots,b_{n-1},b^*),
\end{align*}
namely $(a_1,\dots,a_n)$ majorizes $(b_1,\dots,b_{n-1},b^*)$, where $b^* = \sum_{i=1}^n a_n - \sum_{i=1}^{n-1} b_i$.
\end{lemma}

For a background on majorization see \cite{Marshall2011Inequalities}.

\begin{lemma}[Lemma 3.2 of \cite{Borcea2010Multivariate}]\label{Lem:BivariateRealStability}
Let $p(z_1,z_2) = a_{11}z_1z_2 + a_{10}z_1 + a_{01}z_2 + a_{00} \in \R[z_1,z_2] \backslash \{0\}$. Then, $p \in \Hcal_2(\R)$ if and only if $\det [a_{ij}] \leq 0$.
\end{lemma}

Recall that the class of strongly Rayleigh processes is closed under conditioning and projection.

\begin{lemma}\label{Lem:KernelInterlacing}
Let $\X$ be a strongly Rayleigh process on $[n]$ with kernel polynomial $g$. Denote the kernel polynomials of $\big(\X \cap [n-1] \big| n \in \X\big)$ and $\big(\X \cap [n-1] \big| n \not\in \X\big)$ by $g_1$ and $g_0$, respectively. Then, $\ovl{g}_1 \ll \ovl{g}$ and $\ovl{g}_0 \ll \ovl{g}$.
\end{lemma}

\begin{proof}
Let $\lambda = \lambda(\ovl{g})$, $\gamma = \lambda(\ovl{g}_1)$ and $\delta = \lambda(\ovl{g}_0)$. By \cref{Th:ClassificationOfKernel}, $\lambda_i \in [0,1]$ for all $i \in [n]$. First, we consider the case where $\lambda_i \in (0,1)$ for all $i \in [n]$. Let $f$ be the probability generating polynomial of $\X$. Denote the probability generating polynomials of $\big(\X \cap [n-1] \big| n \in \X\big)$ and $\big(\X \cap [n-1] \big| n \not\in \X\big)$ by $f_1$ and $f_0$, respectively. Note that $f = z_n (p_n f_1) + (1-p_n) f_0$ and $f_1,f_0 \in \R[z_1,\dots,z_{n-1}]$.

By \cref{Prop:ExamplesOfInterlacing} we have $f_1 \ll f$ which implies $\ovl{f}_1 \ll \ovl{f}$. On the other hand,
\begin{align*}
\ovl{f}(x) &= (1-x)^n \, \ovl{g} \bigg(\dfrac{1}{1-x}\bigg) = (\lambda_1 x + 1-\lambda_1) \dots (\lambda_n x + 1-\lambda_n), \\
\ovl{f}_1(x) &= (1-x)^n \, \ovl{g}_1 \bigg(\dfrac{1}{1-x}\bigg) = (\gamma_1 x + 1-\gamma_1 \big) \dots (\gamma_{n-1} x + 1-\gamma_{n-1}).
\end{align*}

Since $\lambda_i > 0$ for all $i \in [n]$, the polynomial $\ovl{f}$ is of degree $n$ and all its roots are negative. Since $\ovl{f}_1 \ll \ovl{f}$ and $\deg(\ovl{f}_1) < \deg(\ovl{f})$, the polynomial $\ovl{f}_1$ is of degree $n-1$ and its roots interlace the roots of $\ovl{f}$. In particular, $\gamma_i > 0$ for all $i \in [n-1]$. Now, since the function $(x-1)/x$ is increasing on $\R_{>0}$, it follows that $\gamma$ interlaces $\lambda$ which implies $\ovl{g}_1 \ll \ovl{g}$.

Now we prove $\ovl{g}_0 \ll \ovl{g}$. For every multi-affine polynomial $p \in \R[z_1,\dots,z_n]$ define
\begin{align*}
\Rcal_n(p) = z_1 \dots z_n \, p \bigg(\dfrac{1}{z_1},\dots,\dfrac{1}{z_n}\bigg).
\end{align*}
By part 4 of \cref{Prop:R.S.P.Properties}, if $p \in \Hcal_n(\R)$ then $\Rcal_n(p) \in \Hcal_n(\R)$. Since $\Rcal$ is linear, we have $\Rcal_n(f) =  p_n \Rcal_{n-1}(f_1) + z_n \big( (1-p_n) \Rcal_{n-1}(f_0) \big)$. Therefore, $\Rcal_{n-1}(f_0) \ll \Rcal_n(f)$ which implies $\ovl{\Rcal_{n-1}(f_0)} \ll \ovl{\Rcal_n(f)}$. On the other hand,
\begin{align*}
\ovl{\Rcal_n(f)}(x) &= x^n \, \ovl{f} \bigg(\dfrac{1}{x}\bigg) = \big( (1-\lambda_1)x + \lambda_1 \big) \dots \big( (1-\lambda_n)x + \lambda_n \big), \\
\ovl{\Rcal_{n-1}(f_0)}(x) &= x^{n-1} \, \ovl{f}_0 \bigg(\dfrac{1}{x}\bigg) = \big( (1-\delta_1)x + \delta_1 \big) \dots \big( (1-\delta_{n-1})x + \delta_{n-1} \big).
\end{align*}
Using the assumption $\lambda_i < 1$, for all $i \in [n]$, and an argument similar to the one used above, we can deduce $\ovl{g}_2 \ll \ovl{g}$.

For the general case, we can approximate every kernel polynomial by kernel polynomials whose diagonalizations have roots in $(0,1)$. For example, this can be achieved using the polynomials $(1+2\varepsilon)^{-n} \, g\big((1+2\varepsilon) \zv - \varepsilon \1\big)$, where $\varepsilon > 0$.
\end{proof}

\begin{proof}[Proof of \cref{Th:EntropyLowerBound}]
We use induction on $n$. The base is $n=2$. Denote the probability generating polynomial of $\X$ by $f_\X$ and let $f_\X(z_1,z_2) = a_0+a_1z_1+a_2z_2+a_3z_1z_2$. Recall that $\ovl{f}_\X(x) = (\lambda_1 x + 1-\lambda_1)(\lambda_2 x + 1-\lambda_2)$. Therefore,
\begin{align*}
a_0+(a_1+a_2)x+a_3x^2 = (1-\lambda_1)(1-\lambda_2) + \big(\lambda_1(1-\lambda_2)+\lambda_2(1-\lambda_1)\big)x + \lambda_1\lambda_2x^2.
\end{align*}
By comparing the coefficients,
\begin{align*}
a_0 &= (1-\lambda_1)(1-\lambda_2), \\
a_1+a_2 &= \lambda_1(1-\lambda_2)+\lambda_2(1-\lambda_1), \\
a_3 &= \lambda_1\lambda_2.
\end{align*}
By \cref{Lem:BivariateRealStability} we have $a_1a_2 \geq a_0a_3 = \big((1-\lambda_1)\lambda_1\big) \big((1-\lambda_2)\lambda_2\big)$. This implies
\begin{align*}
\big((1-\lambda_1)(1-\lambda_2) , \lambda_1(1-\lambda_2) , \lambda_2(1-\lambda_1) , \lambda_1\lambda_2\big) \succ (a_0 , a_1 , a_2 , a_3).
\end{align*}
Since entropy is a Schur-concave function, namely it is non-increasing with respect to majorization (See \cite{Marshall2011Inequalities}),
\begin{align*}
H(\X) = H(a_0,a_1,a_2,a_3) \geq H\big((1-\lambda_1)(1-\lambda_2) , \lambda_1(1-\lambda_2) , \lambda_2(1-\lambda_1) , \lambda_1\lambda_2\big) = H(I_1,I_2).
\end{align*}

Now assume that the statement is true for $n-1$. Let $\X' = \X \cap [n-1]$ and denote the non-increasing vectors of the roots of the kernel polynomials of $(\X'| n \in \X)$ and $(\X'| n \not \in \X)$ by $\gamma$ and $\delta$, respectively. By \cref{Lem:KernelInterlacing}, $\gamma$ and $\delta$ both interlace $\lambda$, where $\lambda$ is the non-increasing vector of the roots of $\ovl{g}_\X$.

%

By the induction hypothesis,
\begin{align*}
H \big( \X' | n \in \X \big) \geq \sum_{i=1}^{n-1} h(\gamma_i) \qquad \text{and} \qquad H \big( \X' | n \not \in \X \big) &\geq \sum_{i=1}^{n-1} h(\delta_i).
\end{align*}
Therefore,
\begin{align}\label{Eq:EntropyLowerBound1}
H(\X) &= p_n H \big( \X' | n \in \X \big) + (1-p_n) H \big( \X' | n \not \in \X \big) + h(p_n) \nonumber \\
&\geq p_n \bigg( \sum_{i=1}^{n-1} h(\gamma_i) \bigg) + (1-p_n) \bigg( \sum_{i=1}^{n-1} h(\delta_i) \bigg) + h(p_n).
\end{align}

Let $\alpha = \sum_{i=1}^n \lambda_i - \sum_{i=1}^{n-1} \gamma_i$ and $\beta = \sum_{i=1}^n \lambda_i - \sum_{i=1}^{n-1} \delta_i$. Since $\gamma$ and $\delta$ both interlace $\lambda$, we have $\alpha \geq 0$ and $\beta \geq 0$. Note that
\begin{align*}
\sum_{i=1}^{n-1} \gamma_i = \E\big[|\X'| \big| n \in \X\big] = \E\big[|\X| \big| n \in \X\big] -1 \qquad \text{and} \qquad \sum_{i=1}^{n-1} \delta_i = \E\big[|\X'| \big| n \not \in \X\big] = \E\big[|\X| \big| n \not \in \X\big].
\end{align*}
Therefore,
\begin{align}\label{Eq:EntropyLowerBound2}
p_n \alpha + (1-p_n) \beta = p_n.
\end{align}
Since $\beta \geq 0$, it follows from the above equation that $\alpha \leq 1$. Also, by negative dependence, $\E\big[|\X'| \big| n \not\in \X\big] \geq \E\big[|\X'|\big]$ which implies that
\begin{align*}
\beta = \E\big[|\X|\big] - \E\big[|\X'| \big| n \not \in \X\big] \leq \big( \E\big[|\X'|\big] + p_n \big) - \E\big[|\X'|\big] \leq 1.
\end{align*}
Thus $h(\alpha)$ and $h(\beta)$ are well-defined.

Now, by \cref{Lem:FromMarshall2}, we have $\lambda \succ (\gamma,\alpha)$ and $\lambda \succ (\delta,\beta)$. Since the entropy function is concave and sum of concave functions is Schur-concave,
\begin{align*}
\sum_{i=1}^{n-1} h(\gamma_i) + h(\alpha) &\geq \sum_{i=1}^n h(\lambda_i), \\
\sum_{i=1}^{n-1} h(\delta_i) + h(\beta) &\geq \sum_{i=1}^n h(\lambda_i).
\end{align*}
Therefore,
\begin{align}\label{Eq:EntropyLowerBound3}
p_n \bigg( \sum_{i=1}^{n-1} h(\gamma_i) \bigg) + (1-p_n) \bigg( \sum_{i=1}^{n-1} h(\delta_i) \bigg) + \Big( p_n h(\alpha) + (1-p_n) h(\beta) \Big) \geq \sum_{i=1}^n h(\lambda_i).
\end{align}
Also, by \eqref{Eq:EntropyLowerBound2} and concavity of $h$ we have
\begin{align}\label{Eq:EntropyLowerBound4}
h(p_n) \geq p_n h(\alpha) + (1-p_n) h(\beta).
\end{align}
The result follows from \eqref{Eq:EntropyLowerBound1}, \eqref{Eq:EntropyLowerBound3} and \eqref{Eq:EntropyLowerBound4}.
\end{proof}

\begin{Rem}
It is possible to prove a stronger result for determinantal processes. Let $\Y$ be a determinantal process on $[n]$ with kernel $K$ and $\lambda_1,\dots,\lambda_n$ be the eigenvalues of $K$. If $I_1,\dots,I_n$ are independent Bernoulli variables with $I_i \sim \mathrm{Bernoulli}(\lambda_i)$, then the distribution of $\Y$ is majorized by the distribution $(I_1,\dots,I_n)$, namely
\begin{align}\label{Eq:DeterminantalMajorization}
\big( \lambda^A(1-\lambda)^{A^c} : A \subseteq [n] \big) \succ \big( \prb(\Y = A) : A \subseteq [n] \big),
\end{align}
where $\lambda = (\lambda_1,\dots,\lambda_n)$. Since entropy is Schur-Concave, this result implies \cref{Th:EntropyLowerBound} in the case of determiantal processes.

The proof of this result relies on two important properties of determinantal point processes. To avoid digression, we will only present a sketch of the proof. By the spectral decomposition, $K = \sum_{i=1}^n \lambda_i \, v_iv_i^*$, where $v_1,\dots,v_n$ are orthonormal. Define $K_I = \sum_{i=1}^n I_i \, v_iv_i^*$ and let $\X_I$ be the (random) determinantal process with kernel $K_I$. \cite[Theorem 4.5.3]{Hough2009Zeros} states that $\X_I \sim \X$. This implies
\begin{align*}
\big( \prb(\Y = A) : A \subseteq [n] \big) = \big( \lambda^A(1-\lambda)^{A^c} : A \subseteq [n] \big) \, M,
\end{align*}
where $M$ is a $\binom{n}{2} \times \binom{n}{2}$ matrix with $M(A,B) = \prb(\X_I=B \, | \, I=A)$ for $A,B \subseteq [n]$. Now, \eqref{Eq:DeterminantalMajorization} holds if and only if $M$ is doubly stochastic (see \cite{Marshall2011Inequalities}). Note that $\sum_{B \subseteq [n]} M(A,B) =1$ for each $A \subseteq [n]$. It remains to show that $\sum_{A \subseteq [n]} \prb(\X_I = B \, | \, I=A) = 1$ for every $B \subseteq [n]$. For each $B \subseteq [n]$, the point process $[X_I \, | \, I=A]$ is a ``determinantal projection process". There is a nice geometric interpretation for the law of such point processes (see \cite{Hough2009Zeros}). Using this interpretation, the desired equation becomes equivalent to the generalization of the Pythagorean theorem to higher dimensions.
\end{Rem}

We expect that the above result also holds for strongly Rayleigh processes.

\begin{Conj}\label{Conj:MajorizationConjecture}
Let $\X$ be a strongly Rayleigh process on $[n]$ with kernel $g_\X$ and $\lambda_1,\dots,\lambda_n$ be the roots of $\ovl{g}_\X$. Assuming $\lambda = (\lambda_1,\dots,\lambda_n)$, we have
\begin{align*}
\big( \lambda^A(1-\lambda)^{A^c} : A \subseteq [n] \big) \succ \big( \prb(\X = A) : A \subseteq [n] \big)
\end{align*}
\end{Conj}

The above conjecture can be regarded as a first step in generalizing \cite[Theorem 4.5.3]{Hough2009Zeros}. This conjecture is equivalent to the existence of a doubly stochastic $\binom{n}{2} \times \binom{n}{2}$ matrix $M$ such that 
\begin{align*}
\big( \prb(\X = A) : A \subseteq [n] \big) = \big( \lambda^A(1-\lambda)^{A^c} : A \subseteq [n] \big) \, M.
\end{align*}
A full description of the entries of this matrix will lead to a generalization of \cite[Theorem 4.5.3]{Hough2009Zeros} to strongly Rayleigh processes.

\subsection{Proof of the Paving Property for Strongly Rayleigh Processes}\label{Subsec:StronglyRayleighPaving}

In this subsection we will prove \cref{Th:S.R.P.PavingEntropyVersion}. The following is a corollary of \cref{Prop:ZeroDiagonalPolynomialPaving}.

\begin{Cor}\label{Cor:S.R.P.PavingKernelVersion}
For every positive $\varepsilon$, there is an integer $r$ such that for any strongly Rayleigh process $\X$ on any space $S$, it is possible to partition $S$ into $r$ subsets $S_1,\dots,S_r$ such that
\begin{align*}
\forall i \in [r] \ : \ \M ( \ovl{\xi}_i ) \leq \varepsilon , \;\quad \xi_i(\zv) := g_{\X \cap S_i}(\zv + p),
\end{align*}
where $p=(p_j)_{j \in S}$ and $p_j = \prb(j \in \X)$.
\end{Cor}

\begin{proof}
Without loss of generality we can assume $S=[n]$. Define $\xi(\zv) = g_\X(\zv + p)$. We claim that $\xi$ satisfies the assumptions of \cref{Prop:ZeroDiagonalPolynomialPaving} with $\Lambda=1$. It is straightforward to verify that $\xi$ is multi-affine real stable, $[z_1 \dots z_n]_\xi = 1$ and $[z_1 \dots z_{i-1} z_{i+1} \dots z_n]_\xi = 0$ for all $i \in [n]$. Now, we must prove that $\M( \ovl{\xi} ) \leq 1$. By \cref{Th:ClassificationOfKernel}, we have $\lambda_i(g_\X) \in [0,1]$ for $i=1,\dots,n$. Therefore, $g_\X$ satisfies the assumptions of \cref{Lem:AboveRootsOfg} and thus $b\1 \in \ab_{g_\X}$ for every $b>1$. Let $b>1$ and $u \geq b\1$. Since $p \geq 0$, we have $u+p \geq b\1$ and so $\xi(u) = g_\X(u+p) \neq 0$. Therefore, $b\1 \in \ab_\xi$ for $b>1$, which implies that $\lambda_i(\,\ovl{\xi}\,) \leq 1$ for $i=1,\dots,n$.

By applying \cref{Lem:AboveRootsOfg} to $(-1)^n g_\X(1-\zv)$ we get $b\1 \in \ab_{g_\X(-\zv)}$ for every $b>0$. Now, we claim that if $b>1$, then $b\1 \in \ab_{\xi(-\zv)}$. Let $u \geq b\1$. Since $p \leq 1$, we have $u-p \geq 0$ and so $\xi(-u) = g_\X(-u+p) = g_\X\big(-(u-p)\big) \neq 0$. This proves our claim, which implies that $\lambda_i(\ovl{\xi}) \geq -1$ for $i=1,\dots,n$.

We showed that all the roots of $\ovl{\xi}$ lie in $[-1,1]$. Now, apply \cref{Prop:ZeroDiagonalPolynomialPaving} to $\xi$ and choose a large $r$.
\end{proof}

In order to deduce \cref{Th:S.R.P.PavingEntropyVersion} from the above result, it is sufficient to prove the following.

\begin{Prop}\label{Prop:ProbabilisticInterpretation}
Assume that $\X$ is a strongly Rayleigh process on $[n]$ with kernel $g_\X$. Set $p_i = \prb(i \in \X)$ and $\xi(z_1,\dots,z_n) = g_\X(z_1+p_1,\dots,z_n+p_n)$. For every positive $\delta$, there exists a positive $\varepsilon$ such that if all the roots of $\ovl{\xi}$ have absolute value less than $\varepsilon$, then
\begin{align*}
\bigg| \dfrac{1}{n} H(\X) - \dfrac{1}{n} \sum_{i=1}^n h(p_i) \bigg| < \delta.
\end{align*}
\end{Prop}

We will use majorization properties of hyperbolic polynomials. A homogeneous polynomial $p \in \R[z_1,\dots,z_n]$ is \textit{hyperbolic with respect to} a vector $e \in \R^n$ if $p(e)>0$ and $p(te+\alpha) \in \R[t]$ is real rooted for all $\alpha \in \R^n$. We use $\lambda_\alpha(p)$ to denote the vector of roots of the polynomial $p(te+\alpha)$ in the non-increasing order. The following theorem is proved in \cite{Gurvits2004Combinatorics}.

\begin{Th}\label{Th:HyperbolicMajorization}
Let $p \in \R[z_1,\dots,z_n]$ be hyperbolic with respect to $e$. For $v,u \in \R^n$ we have
\begin{align*}
\lambda_{v+u}(p) \prec \lambda_v(p) + \lambda_u(p).
\end{align*}
\end{Th}

We will use the following lemma in the proof of \cref{Prop:ProbabilisticInterpretation}.

\begin{lemma}\label{Lem:ProofOfProbabilisticInterpretation}
Let $p_\downarrow$ with $p_\downarrow = \big(p_{(1)},\dots,p_{(n)}\big)$, be the vector of $p_i$'s in the non-increasing order. We have
\begin{align*}
\lambda(\ovl{g}_\X) \prec \lambda(\ovl{\xi}) + p_\downarrow.
\end{align*}
\end{lemma}

\begin{proof}
Assume $\xi(\zv) = \sum_{A \subseteq [n]} b_A \, \zv^{A^c}$. Since $g_\X(z_1,\dots,z_n) = \xi(z_1-p_1,\dots,z_n-p_n)$, we have
\begin{align*}
g_\X(\zv) = \sum_{A \subseteq [n]} \Bigg( \sum_{B \subseteq A} b_B (-p)^{A \backslash B} \Bigg) \, \zv^{A^c}.
\end{align*}
Define the polynomial $F \in \R[z_1,\dots,z_n,u_1,\dots,u_n,w]$ as
\begin{align*}
F(\zv,\mathbf{u},w) = \sum_{A \subseteq [n]} \Bigg( \sum_{B \subseteq A} b_B \, w^{|B|} \, \mathbf{u}^{A \backslash B} \Bigg) \, \zv^{A^c},
\end{align*}
where $\mathbf{u} = (u_1,\dots,u_n)$ and $\zv = (z_1,\dots,z_n)$. We claim that $F$ is hyperbolic with respect to $e \in \R^{2n+1}$, where $e_1=\dots=e_n=1$ and $e_{n+1}=\dots=e_{2n+1}=0$. Let $\bld{z} = (\zv,\mathbf{u},w) \in \R^{2n+1}$. If $w \neq 0$, then
\begin{align*}
F(te+\bld{z}) = w^n \xi \bigg( \dfrac{1}{w}(t+z_1+u_1),\dots,\dfrac{1}{w}(t+z_1+u_1) \bigg).
\end{align*}
Since $\xi$ is real stable and $u_1,\dots,u_n,w \in \R$, the above polynomial is real rooted. If $w = 0$, then
\begin{align*}
F(te+\bld{z}) = \sum_{A \subseteq [n]} \mathbf{u}^A (\bld{t}+\zv)^{A^c} = \prod_{i=1}^n (t+z_i+u_i),
\end{align*}
where $\bld{t} = (t,\dots,t)$. This polynomial is also real rooted and our claim follows.

Now, by \cref{Th:HyperbolicMajorization},
\begin{align*}
\lambda_{(\0,-p,1)}(F) \prec \lambda_{(\0,\0,1)}(F) + \lambda_{(\0,-p,0)}(F),
\end{align*}
where $\0 = (0,\dots,0) \in \R^n$. It is straightforward to verify that
\begin{align*}
\lambda_{(\0,-p,1)}(F) = \lambda(\ovl{g}_\X), \qquad \lambda_{(\0,\0,1)}(F) = \lambda(\ovl{\xi}), \qquad \lambda_{(\0,-p,0)}(F) = p_\downarrow.
\end{align*}
The result follows.
\end{proof}

Now we are ready to prove \cref{Prop:ProbabilisticInterpretation}.

\begin{proof}[Proof of \cref{Prop:ProbabilisticInterpretation}]
Let $\gamma_1,\dots,\gamma_n$, indexed in non-increasing order, be the roots of $\ovl{\xi}$ and $\lambda_1,\dots,\lambda_n$, indexed in non-increasing order, be the roots of $\ovl{g}_\X$. Since the entropy function is concave and sum of concave functions is Schur-concave, it follows from \cref{Lem:ProofOfProbabilisticInterpretation} that
\begin{align}\label{Eq:ProofOfProbabilisticInterpretation1}
\sum_{i=1}^n h\big(\gamma_i + p_{(i)}\big) \leq \sum_{i=1}^n h(\lambda_i).
\end{align}
Choose $\varepsilon$ such that if $|x-y| < \varepsilon$, then $\big|h(x)-h(y)\big| < \delta$. Therefore, if $|\gamma_i| < \varepsilon$ for all $i \in [n]$, then by \eqref{Eq:ProofOfProbabilisticInterpretation1},
\begin{align}\label{Eq:ProofOfProbabilisticInterpretation2}
\sum_{i=1}^n h(\lambda_i) \geq \sum_{i=1}^n h\big(\gamma_i + p_{(i)}\big) \geq \sum_{i=1}^n h(p_i) - n\delta.
\end{align}
On the other hand, by \cref{Th:EntropyLowerBound},
\begin{align}\label{Eq:ProofOfProbabilisticInterpretation3}
\sum_{i=1}^n h(\lambda_i) \leq H(\X) \leq \sum_{i=1}^n h(p_i).
\end{align}
Combining \eqref{Eq:ProofOfProbabilisticInterpretation1}, \eqref{Eq:ProofOfProbabilisticInterpretation2} and \eqref{Eq:ProofOfProbabilisticInterpretation3} we get
\begin{align*}
\sum_{i=1}^n h(p_i) - n\delta \leq H(\X) \leq \sum_{i=1}^n h(p_i).
\end{align*}
This completes the proof.
\end{proof}

\subsubsection*{Acknowledgment}

We would like to express our deepest appreciation to Amir Daneshgar and Mohammadsadegh Zamani for their valuable and instructive comments on an early draft of this paper. We also wish to thank Ziheng Zhu who pointed out to a technical mistake in an earlier version of this paper.

\printbibliography

\appendix

\section{Appendix: Proof of \texorpdfstring{\cref{Lem:AbLemma}}{Lemma 3.8}}\label{App:HyperbolicPolynomials}

We recall some results from the theory of hyperbolic polynomials.

\begin{Def}\label{Def:HyperbolicPolynomials}
A homogeneous polynomial $p \in \R[z_1,\dots,z_n]$ is \textit{hyperbolic with respect to} a vector $e \in \R^n$ if $p(e)>0$ and $p(te+\alpha) \in \R[t]$ is real rooted for all $\alpha \in \R^n$. We use $\lambda_\alpha(p)$ to denote the vector of roots of the polynomial $p(te+\alpha)$ in the non-increasing order.
\end{Def}

Recall that the \textit{homogenization} of a polynomial $p \in \C[z_1,\dots,z_n]$ of degree $d$ is the unique homogeneous polynomial $p_H$ of degree $d$ in the variables $z_1,\dots,z_{n+1}$ such that
\begin{align*}
p_H(z_1,\dots,z_n,1) = p(z_1,\dots,z_n).
\end{align*}
The relationship between real stability and hyperbolicity is made explicit in the following proposition.

\begin{Prop}[Proposition 1.1 of \cite{Borcea2010Multivariate}]\label{Prop:StabilityHyperbolicity}
A polynomial $p \in \R[z_1,\dots,z_n]$ is real stable if and only if its homogenization is hyperbolic with respect to all vectors $e \in \R^{n+1}$ with $e_i>0$ for $i \in [n]$ and $e_{n+1}=0$.
\end{Prop}

Above the roots of a real stable polynomial is akin to the concept of hyperbolicity cone of hyperbolic polynomials.

\begin{Def}
Let $p \in \R[z_1,\dots,z_n]$ be hyperbolic with respect to $e\in\R^n$. The \textit{hyperbolicity cone} of $p$, denoted $C_e(p)$, is $\{x \in \R^n : \ p(x+te) \neq 0 \;\; \text{for} \;\; t \geq 0\}$.
\end{Def}

The following result is due to G\aa rding \cite{Garding1959Inequality}.

\begin{Prop}\label{Prop:HyperbolicityConeProperties}
Let $p \in \R[z_1,\dots,z_n]$ be hyperbolic with respect to $e \in \R^n$. Then
\begin{enumerate}
\item
$C_e(p)$ is convex;

\item
$C_e(p)$ is equal to the connected component of the set $\{x \in \R^n : p(x) \neq 0\}$ that contains $e$.

\item
$p$ is hyperbolic with respect to any $u \in C_e(p)$ and $C_u(p) = C_e(p)$.
\end{enumerate}
\end{Prop}

The connection between above the roots and hyperbolicity cone is made explicit in the following result which follows immediately from the above proposition.

\begin{Cor}\label{Cor:AbHyperbolicityCone}
Let $p \in \Hcal_n(\R)$ and $p_H$ be its homogenization. For every $e \in \R_+^n$ we have
\begin{align*}
\ovl{\ab_p} \times \{1\} = \ovl{C_{(e,0)}(p_H)} \cap \{z_{n+1} = 1\},
\end{align*}
where $\ovl{U}$ denotes the closure of set $U$.
\end{Cor}

\begin{proof}
By \cref{Prop:StabilityHyperbolicity}, $p_H$ is hyperbolic with respect to $(e,0)$ for every $e \in \R_+^n$. Now, the result follows since
\begin{align*}
\ovl{\ab_p} &= \{x \in \R^n : p(y) \neq 0 \;\; \text{for} \;\; y>x\}, \\
\ovl{C_{(e,0)}(p_H)} &= \{x \in \R^{n+1} : \ p_H(x+te) \neq 0 \;\; \text{for} \;\; t > 0\}.
\end{align*}
\end{proof}

The boundary of above the roots is characterized in the following lemma.

\begin{lemma}\label{Lem:BoundaryOfAb}
Let $p \in \Hcal_n(\R)$. If $u \in \ovl{\ab_p}$ and $p(u) \neq 0$, then $u \in \ab_p$.
\end{lemma}

\begin{proof}
We must show that $p(u+v) \neq 0$ for every $v \in \R_{\geq 0}^n$. Assume otherwise and let $v \in \R_{\geq 0}^n$ be such that $p(u+v) = 0$. We claim that $p(u+tv) = 0$ for all $t>0$. Assume to the contrary that $t>0$ and $p(u+tv) \neq 0$.

Note that it follows from the proof of \cref{Cor:AbHyperbolicityCone} that if $\alpha \in \ovl{\ab_p}$ and $p(\alpha) \neq 0$, then $(\alpha,1) \in C_{(e,1)}(p_H)$ for every $e \in \R_+^n$. Therefore, by the assumptions, $(u,1) \in C_{(e,0)}(p_H)$. Also, since $p(u+tv) \neq 0$, we have $(u+tv,1) \in C_{(e,0)}(p_H)$.

Assume that $t \geq 1$. Since $C_{(e,0)}(p_H)$ is convex by part 1 of \cref{Prop:HyperbolicityConeProperties}, we have $(u+v,1) \in C_{(e,0)}(p_H)$. Consequently, $p(u+v) \neq 0$ which is a contraction. Therefore, we have $p(u+tv) = 0$ for all $t \geq 1$.

Now, assume that $t \in (0,1)$. Let $w \in \R_{\geq0}^n$ be a vector such that $w_i = 0$ for those $i$ that $v_i \neq 0$ and $w_i < 0$ for those $i$ that $v_i = 0$. Assume that $w$ has small Euclidean norm. Then, by continuity of roots with respect to coefficients, $p(u+t(v+w)) \neq 0$ and hence $(u+t(v+w),1) \in C_{(e,0)}(p_H)$. Let $s>1$ and consider the following vector
\begin{align*}
a := \big(u+t(v+w)\big) + s\Big((u+v)-\big[u+t(v+w)\big]\Big) = u + (t+s-ts)v + t(1-s) w.
\end{align*}
Note that $a > u$. Hence $(a,1) \in C_{(e,0)}(p_H)$. Also, $u+v$ is on the segment connecting $u+t(v+w)$ and $a$. Therefore, since $C_{(e,0)}(p_H)$ is convex, $(u+v,1) \in C_{(e,0)}(p_H)$ and so $p(u+v) \neq 0$ which is a contradiction. Therefore, we have $p(u+tv) = 0$ for all $0 < t < 1$. This completes the proof of our claim. But now it follows that $p(u) = 0$, which itself is a contradiction. So we must have $p(u+v) \neq 0$ for every $v \in \R_{\geq 0}^n$ which means $u \in \ab_p$.
\end{proof}

We are ready to prove \cref{Lem:AbLemma}.


\begin{proof}[Proof of \cref{Lem:AbLemma}]
The assumption is equivalent to $p((u-v)+tv) \neq 0$ for all $t \in [0,1]$. Let $w \in \R_{\geq0}^n$ be a vector with small Euclidean norm such that $w_i=0$ for those $i$ that $v_i \neq 0$ and $v+w \in \R_+^n$. By continuity of roots with respect to coefficients, $p((u-v)+t(v+w)) \neq 0$ for all $t \in [0,1]$. For $t>1$ we have
\begin{align*}
p\big((u-v)+t(v+w)\big) = p\big(u+((t-1)v+tw)\big) \neq 0,
\end{align*}
where we used the fact that $(t-1)v+tw \in \R_+^n$ and $u \in \ab_p$. Therefore, for all $t \geq 0$ we have $p(u-v+t(v+w)) \neq 0$. Also, by \cref{Prop:StabilityHyperbolicity}, $p_H$ is hyperbolic with respect to $(v+w,0)$. Therefore, $(u-v,1) \in C_{(v+w,0)}(p_H)$. By \cref{Cor:AbHyperbolicityCone}, we have $u-v \in \ovl{\ab_p}$. Since $p(u-v) \neq 0$, it follows from \cref{Lem:BoundaryOfAb} that $u-v \in \ab_p$.
\end{proof}

\end{document}